%% file: banff.tex
\documentclass[11pt]{article}

\input{preamble.tex}

\title{Cyclic polytopes, orientals, and correspondences: some aspects of higher Segal spaces} 
\author{Tobias Dyckerhoff\footnote{Universität Hamburg
Fachbereich Mathematik
Bundesstraße 55
20146 Hamburg, email: {\tt tobias.dyckerhoff@uni-hamburg.de}}}

\begin{document}

\maketitle

\begin{abstract}
    We discuss the role of higher Segal spaces at the interface of cyclic
    polytopes, orientals, and higher correspondences. Along the way
    we review examples from algebraic K-theory, show how cyclic polytopes
    provide a geometric model for the definition of orientals, and establish a
    characterization of higher Segal spaces as lax monadic structures in higher
    correspondence categories.
\end{abstract}

\tableofcontents

\section{Introduction}

These are extended notes on talks given at the Banff workshop on ``Higher Segal
Spaces and their Applications to Algebraic K-Theory, Hall Algebras, and
Combinatorics''. Many of the results and ideas presented here are in some form
present in the literature while we try to tie some loose ends or discuss
slightly different perspectives here and there.
We give a rough overview of the contents:

\begin{itemize}
    \item[\S \ref{sec:cyclic_polytopes}.] Discussion of the geometry and combinatorics of cyclic polytopes
    \item[\S \ref{sec:higher_segal_objects}.]  Definition, basic theory, and examples of higher Segal spaces
    \item[\S \ref{sec:orientals}.]  Geometric construction of Street's orientals \cite{street} via cyclic polytopes
    \item[\S \ref{sec:the_interplay_of_higher_segal_conditions}.]  Some key
        results on higher Segal spaces based on results of Rambau \cite{rambau} on cyclic
        polytopes
    \item[\S \ref{sec:correspondences}.] Description of the
        $(\infty,\omega)$-category of correspondences via barycentric
        subdivision
    \item[\S \ref{sec:monads}.] Characterization of higher Segal spaces as
        certain lax monads in higher correspondence categories
    \item[\S \ref{sec:higherex}.] Alternative characterization of higher Segal spaces via higher excision due to T. Walde 
    \item[\S \ref{sec:further}.] Some further perspectives 
\end{itemize}

{\bf Acknowledgements.} I thank the organizers of the workshop, namely Julie
Bergner, Joachim Kock, and Maru Sarazola, for their initiative to bring
together the community and for creating a fruitful atmosphere for the many
insightful exchanges that happened during the workshop. I am indebted to
Mikhail Kapranov for all the shared insights over the years -- in particular,
his proposal to use cyclic polytopes as a foundation for the theory of higher
Segal spaces is a cornerstone of everything presented in this article. 
I am very grateful to an anonymous referee for the careful reading and the many
suggestions that helped improve the text. Last but not least, I acknowledge
support by the Deutsche Forschungsgemeinschaft (DFG, German Research
Foundation) under Germany's Excellence Strategy - EXC 2121 ``Quantum Universe''
- 390833306 and the Collaborative Research Center - SFB 1624 ``Higher
structures, moduli spaces and integrability'' - 506632645.

\section{Cyclic polytopes}%
\label{sec:cyclic_polytopes}

We recall some basics of cyclic polytopes. See \cite{rambau}, for more
background and detailed proofs.

\begin{defi}
\label{def:cyclic_polytope}
    Let $d \ge 0$ and consider the map
    \[
        \nu: \RR \hra \RR^d, \; t \mapsto (t,t^2,t^3,...,t^d).
    \]
    For a finite subset $S \subset \RR$, we define the $d$-dimensional cyclic polytope 
    \[
        C(S,d) \subset \RR^d
    \]
    on $S$ to be the convex hull of the set $\nu(S) \subset \RR^d$. 
\end{defi}

For any subset $I \subset S$ of cardinality $d+1$, the convex hull of
$\nu(I)$ is a $d$-simplex in $\RR^d$, in other words, the points in
$\nu(S)$ are in general position. This holds, since the determinant
of the van-der-Monde matrix, with rows given by the affine coordinate vectors
\[
    (1, \nu(i)) = (1,\nu(i)_1, \nu(i)_2, ..., \nu(i)_d) 
\]
of the points $\nu(i)$, $i \in M$, is nonzero. In other words, no $d+1$
vertices of $C(S,d)$ lie on an affine hyperplane. This implies that $C(S,d)$ is
a simplicial polytope, i.e. all boundary faces are simplices. 

\begin{rem}
    \label{rem:projection} It is a key feature that cyclic polytopes on the
    same vertex set in different dimensions are related by projection maps: The
    map
    \[
        \pi: \RR^d \to \RR^{d-1}, (t_1, ..., t_d) \mapsto (t_1, ..., t_{d-1})
    \]
    maps $C(S,d)$ onto $C(S,d-1)$ (cf. Figure \ref{fig:cyclicseries}).
\end{rem}

We will typically study the cyclic polytope on the set
\[
    S = [n] = \{0,1,...,n\} \subset \RR
\]
for varying $n \ge 0$. The cyclic polytopes $C([n],d)$ are simplicial
polytopes, i.e. their boundary is a simplicial complex. We introduce some
terminology and notation to speak about simplicial complexes.

\begin{defi}
    \label{defi:abstract}
    Given a set $S$, we define an {\em abstract simplicial complex $K$} on $S$ to
    be a subset of the set $\P^*(S)$ of nonempty finite subsets of $S$ satisfying the
    condition:
    \begin{itemize}
        \item if $\sigma \in K$ and $\emptyset \neq \tau \subset \sigma$, then $\tau \in K$.
    \end{itemize}
    For any arbitrary subset $K \subset \P^*(S)$ we define $\langle K
    \rangle \subset \P^*(S)$ to be the smallest abstract simplicial complex
    containing $K$, obtained by adjoining all nonempty subsets of the sets in
    $K$. 
\end{defi}

Let $I \subset [n]$ be a subset of cardinality $d$. We discuss a simple
criterion to decide whether the $(d-1)$-simplex $|\Delta^I|$ given by the
convex hull of $\nu(I)$ is a boundary simplex of $C([n],d)$. Namely, this will
be the case, if all points in $\nu([n] \setminus I)$ lie on the same side of
the affine hyperplane $H_I \subset \RR^d$ spanned by $\nu(I)$. We derive a
beautiful criterion for testing this. First, note that the projection $\pi$
determines a canonical labelling of the two sides of the hyperplane $H_I$ as
follows. Consider the fiber of $\pi$ over a point $x \in \RR^{d-1}$, i.e. a
line of the form 
\[
    L_x = \{x\} \times \RR \subset \RR^d.
\]
Then $L_x$ intersects $H_I$ in precisely one point $(x,t_0)$ so that $\RR^d
\setminus H_I$ splits into an {\em upper half $H_I^{+}$} containing points
$(x,t)$ with $t>t_0$ and a {\em lower half } $H_I^-$ containing points $(x,t)$
with $t < t_0$.  
To determine on which side of $H_I$ a given point $y \in \RR^d$ lies, we form the
van-der-Monde matrix $V$ with rows 
\[
    (1,\nu(i)), \quad i \in I,
\]
ordered according to the linear order of $I$, along with the last row $(1,y)$. Then we have
\begin{enumerate}[label=(\alph *)]
    \item $\det(V) > 0$ iff $y \in H_I^+$,
    \item $\det(V) = 0$ iff $y \in H_I$,
    \item $\det(V) < 0$ iff $y \in H_I^-$.
\end{enumerate}
If $y \in H_I^+$ (resp. $y \in H_I^-$), we say $y$ {\em lies above (resp.
below) $H_I$}. Finally, combining this observation with the fact that the
determinant changes sign when swapping two rows, we arrive at {\em Gale's
eveness criterion} \cite{gale:eveness}:
\begin{enumerate}
    \item a point $\nu(j) \in \RR^d$ corresponding to a {\em gap } $j
        \in [n] \setminus I$ lies above (resp. below) $H_I$ iff the cardinality
        of the set
        \[
            \{i \in I\; | \; i > j \}
        \]
        is even (resp. odd). In this case, we refer to the gap $j$
        itself as {\em even} (resp. {\em odd}). 
    \item the simplex $|\Delta^I|$ is a boundary simplex if 
        \begin{enumerate}
            \item either all gaps $j \in [n] \setminus I$ are even. In this case, we have $C([n],d) \subset H_I \cup H_I^+$,
            \item or all gaps $j \in [n] \setminus I$ are odd. In this case, we have $C([n],d)\subset H_I \cup H_I^-$.
        \end{enumerate}
\end{enumerate}

Therefore, we observe that, due to the canonical notions of {\em above} and
{\em below} derived from the projection map $\pi$, the boundary of the cyclic
polytope $C([n],d)$ decomposes into an upper and lower hemisphere. 

\begin{defi}
\label{def:gale_eveness}
    Let $n \ge 0$. A subset $I \subset [n]$ is called {\em even} (resp. {\em
    odd}) if, for every $j \in [n] \setminus I$, the cardinality of the set
    \[
        \{i \in I \; | \; i > j \}
    \]
    is even (resp. odd). We further introduce the abstract simplicial complexes on $[n]$
    \begin{align*}
        \L([n],d-1) & := \langle J \subset [n] | \text{$J$ even of cardinality $d$} \rangle \subset \P^*([n]) \\
        \U([n],d-1) & := \langle J \subset [n] | \text{$J$ odd of cardinality $d$} \rangle \subset \P^*([n]),
    \end{align*}
    called the {\em lower} and {\em upper boundary} of $C([n],d)$, where we use
    the notation $\langle K \rangle$ from Definition \ref{defi:abstract}.
\end{defi}

For an abstract simplicial complex $K$ on $[n]$, we denote by $|K| \subset
\RR^d$ its $d$-dimensional geometric realization, given as the union of all
convex hulls of the sets $\nu(I)$ for $I \in K$. Typically, we will only
consider the geometric realization when $|K| \subset \RR^d$ is a {\em
geometric simplicial complex}, i.e., any given pair of geometric simplices is
disjoint or intersects in another geometric simplex. Combining the above
observations, we arrive at the following statements. 

\begin{prop}
    \label{prop:features}
    Let $n,d \ge 1$.
    \begin{enumerate}
        \item The geometric realizations $|\L([n],d-1)| \subset \RR^d$ and
            $|\U([n],d-1)| \subset \RR^d$ are geometric simplicial complexes. 
        \item The cyclic polytope $C([n],d)$ is a simplicial polytope with
            boundary given by
            \[
               \partial C([n],d) = |\L([n],d-1)| \cup |\U([n],d-1)| \subset \RR^d
            \]
        \item\label{it:3} The restrictions of the projection map $\pi: \RR^d \to \RR^{d-1}$
            provide homeomorphisms
            \[
                |\U([n],d-1)| \overset{\cong}{\lra} C([n],d-1)
            \]
            and
            \[
                |\L([n],d-1)| \overset{\cong}{\lra} C([n],d-1).
            \]
        \item Under the homeomorphisms from \ref{it:3}, the geometric
            simplicial complexes get identified with triangulations of
            $C([n],d-1)$, again given by the abstract simplicial complexes
            $\L([n],d-1)$ and $\U([n],d-1)$ but geometrically realized in
            $\RR^{d-1}$.     
    \end{enumerate}
\end{prop}

\begin{rem}
    \label{rem:geom_upper_lower}
    Informally speaking, the two triangulations of $C([n],d-1)$ from
    Proposition \ref{prop:features} arise by ``looking'' at the simplicial
    boundary of the polytope $C([n],d)$ from above (resp. below), with respect
    to the direction given by the last coordinate of $\RR^d$ (cf. Figure
    \ref{fig:cyclicseries}). Putting this more formally: For every $x \in
    C([n],d-1)$, the line 
    \[
        \{x\} \times \RR \subset \RR^{d}
    \]
    intersects the polytope $C([n],d)$ in an interval $\{x\} \times [l_x,u_x]$.
    Then the union of the points $\{(x,l_x)\}$ (resp. $\{(x,u_x)\}$) forms the
    geometric realization of the simplicial complex $\L([n],d-1)$ (resp.
    $\U([n],d-1)$).
\end{rem}

\begin{exa}
    \label{exa:pachner}
    Note that we have $C([n],n) = |\Delta^n|$. Thus in this case, Proposition
    \ref{prop:features} yields a decomposition
    \[
        \partial |\Delta^n| = |\L([n],n-1)| \cup |\U([n],n-1)| 
    \]
    of the boundary of the $n$-simplex where $\L([n],n-1)$ and
    $\U([n],n-1)$ correspond to the union of the faces opposite to vertices of
    the same parity. These give rise to two different triangulations of the
    polytope $C([n], n-1)$ depicted for $n=4$ in Figure \ref{fig:3-2}.
\end{exa}

\begin{figure}
\begin{centering}

\begin{tikzpicture}[line join=bevel,z=-5.5]

\node (A) at (8,1,0) {$C([5],3)$};
\node (B) at (8,-3,0) {$C([5],2)$};
\node (C) at (8,-6.5,0) {$C([5],1)$};

\draw[shorten >=0.5cm, shorten <=.5cm, ->] (A) -- (B) node[midway,right] {$\pi$};
\draw[shorten >=0.5cm, shorten <=.5cm, ->] (B) -- (C) node[midway,right] {$\pi$};

\tiny
\begin{scope}[shift={(0,0,0)}]

\coordinate[label=below:{$0$}] (A1) at (0,2,0);
\coordinate[label=below:{$1$}] (A2) at (1,1,0);
\coordinate[label=below:{$2$}] (A3) at (2,0.5,0);
\coordinate[label=below:{$3$}] (A4) at (3,0.5,0);
\coordinate[label=below:{$4$}] (A5) at (4,1,0);
\coordinate[label=below:{$5$}] (A6) at (5,2,0);

\draw [fill opacity=0.8,fill=black!20] (A1) -- (A2) -- (A6) -- cycle;
\draw [fill opacity=0.8,fill=black!40] (A2) -- (A3) -- (A6) -- cycle;
\draw [fill opacity=0.8,fill=black!60] (A3) -- (A4) -- (A6) -- cycle;
\draw [fill opacity=0.8,fill=black!80] (A4) -- (A5) -- (A6) -- cycle;


\begin{scope}[shift={(-2,0,0)}]
\draw[->] (0,0,0) -- (1,0,0) node[anchor=west] {$x$};
\draw[->] (0,0,0) -- (0,1,0) node[anchor=south] {$y$};
\draw[->] (0,0,0) -- (0,0,1) node[anchor=north] {$z$};
\end{scope}

\draw (A1) -- (A2) -- (A6) -- cycle;  
\draw (A2) -- (A3) -- (A6) -- cycle;
\draw (A3) -- (A4) -- (A6) -- cycle;
\draw (A4) -- (A5) -- (A6) -- cycle;

\draw [dashed] (A1) -- (A5);  
\draw [dashed] (A1) -- (A4);  
\draw [dashed] (A1) -- (A3);  

\end{scope}

\begin{scope}[shift={(0,-4,0)}]

\coordinate[label=below:{$0$}] (A1) at (0,2,0);
\coordinate[label=below:{$1$}] (A2) at (1,1,0);
\coordinate[label=below:{$2$}] (A3) at (2,0.5,0);
\coordinate[label=below:{$3$}] (A4) at (3,0.5,0);
\coordinate[label=below:{$4$}] (A5) at (4,1,0);
\coordinate[label=below:{$5$}] (A6) at (5,2,0);

\begin{scope}[shift={(-2,0,0)}]

\draw[->] (0,0,0) -- (1,0,0) node[anchor=west] {$x$};
\draw[->] (0,0,0) -- (0,1,0) node[anchor=south] {$y$};

\end{scope}

\draw [fill opacity=0.7,fill=black!20] (A1) -- (A2) -- (A3) -- (A4) -- (A5) -- (A6) -- cycle;  

\end{scope}

\begin{scope}[shift={(0,-6.5,0)}]

\coordinate[label=below:{$0$}] (A1) at (0,0,0);
\coordinate[label=below:{$1$}] (A2) at (1,0,0);
\coordinate[label=below:{$2$}] (A3) at (2,0,0);
\coordinate[label=below:{$3$}] (A4) at (3,0,0);
\coordinate[label=below:{$4$}] (A5) at (4,0,0);
\coordinate[label=below:{$5$}] (A6) at (5,0,0);

\begin{scope}[shift={(-2,0,0)}]
\draw[->] (0,0,0) -- (1,0,0) node[anchor=west] {$x$};
\end{scope}

\draw (A1) -- (A2) -- (A3) -- (A4) -- (A5) -- (A6);  

\end{scope}

\end{tikzpicture}

\end{centering}
    \caption{The cyclic polytopes $C([5],3)$, $C([5],2)$, and $C([5],1)$, related via the projection maps $\pi$.}
    \label{fig:cyclicseries}
\end{figure}
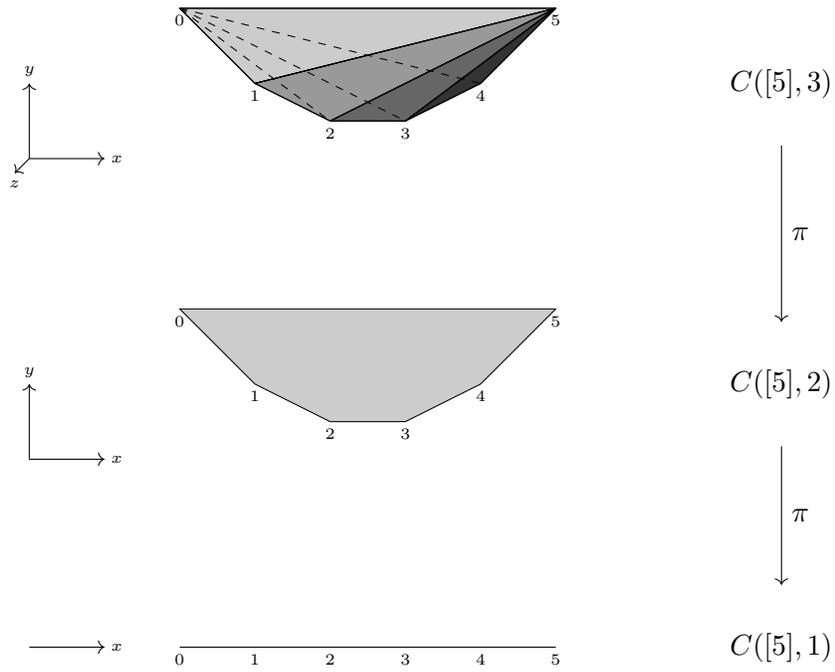

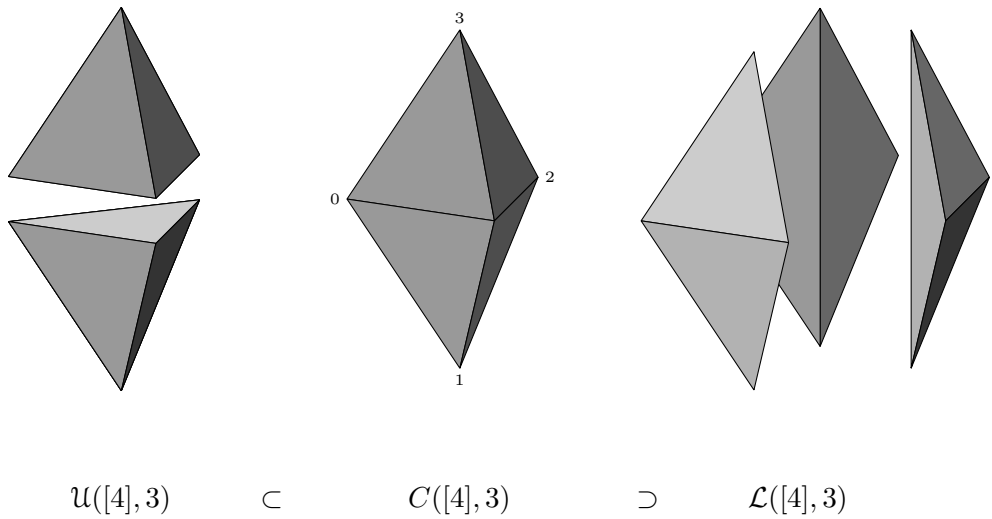
\begin{figure}
    \centering
    \begin{tikzpicture}[line join=bevel,z=-5.5]

         \node (A) at (0,-4,0) {$\U([4],3)$};
         \node (AB) at (2,-4,0) {$\subset$};
         \node (B) at (4.5,-4,0) {$C([4],3)$};
         \node (BC) at (7,-4,0) {$\supset$};
         \node (C) at (9,-4,0) {$\L([4],3)$};

        \tiny

        \begin{scope}[scale=1.5]
            \begin{scope}[shift={(0,0.2,0)}]
            \coordinate (A1) at (0.5,0,-1);
            \coordinate (A2) at (-1,0,0);
            \coordinate (A3) at (.5,0,1);
            \coordinate (N) at (0,1.5,0);
            \draw [dashed] (A1) -- (A2);
            \draw (A2) -- (A3) -- (N) -- cycle;
            \draw (A3) -- (A1) -- (N) -- cycle;
            \draw [fill opacity=0.7,fill=black!40] (A2) -- (A3) -- (N) -- cycle;
            \draw [fill opacity=0.7,fill=black!70] (A3) -- (A1) -- (N) -- cycle;
            \end{scope}

            \begin{scope}[shift={(0,-0.2,0)}]
            \coordinate (A1) at (0.5,0,-1);
            \coordinate (A2) at (-1,0,0);
            \coordinate (A3) at (.5,0,1);
            \coordinate (S) at (0,-1.5,0);
            \draw (A2) -- (S) -- (A1) -- cycle;
            \draw (A3) -- (S) -- (A2) -- cycle;
            \draw (A1) -- (S) -- (A3) -- cycle;
            \draw (A1) -- (A2) -- (A3) -- cycle;
            \draw [fill opacity=0.7,fill=black!20] (A1) -- (A2) -- (A3) -- cycle;
            \draw [fill opacity=0.7,fill=black!40] (A2) -- (A3) -- (S) -- cycle;
            \draw [fill opacity=0.7,fill=black!80] (A3) -- (A1) -- (S) -- cycle;
            \end{scope}

        \end{scope}

        \begin{scope}[scale=1.5, shift={(3,0,0)}]
            \coordinate[label=right:{$2$}] (A1) at (0.5,0,-1);
            \coordinate[label=left:{$0$}] (A2) at (-1,0,0);
            \coordinate[label=right:{$4$}] (A3) at (.5,0,1);
            \coordinate[label={$3$}] (N) at (0,1.5,0);
            \coordinate[label=below:{$1$}] (S) at (0,-1.5,0);
            \draw [dashed] (A1) -- (A2);
            \draw [fill opacity=0.7,fill=black!40] (A2) -- (A3) -- (N) -- cycle;
            \draw [fill opacity=0.7,fill=black!70] (A3) -- (A1) -- (N) -- cycle;
            \draw [fill opacity=0.7,fill=black!40] (A2) -- (A3) -- (S) -- cycle;
            \draw [fill opacity=0.7,fill=black!70] (A3) -- (A1) -- (S) -- cycle;
         \end{scope}

        \begin{scope}[scale=1.5, shift={(6,0,0)}]

            \begin{scope}[shift={(0,0,-1)}]
            \coordinate (A1) at (0.5,0,-1);
            \coordinate (A2) at (-1,0,0);
            \coordinate (A3) at (.5,0,1);
            \coordinate (N)  at (0,1.5,0);
            \coordinate (S)  at (0,-1.5,0);
            \draw [dashed] (A1) -- (A2); 
            \draw [fill opacity=0.8,fill=black!60](N) -- (S) -- (A1) -- cycle;
            \draw [fill opacity=0.8,fill=black!40](N) -- (S) -- (A2) -- cycle;
            \end{scope}

            \begin{scope}[shift={(1,0,0)}]
            \coordinate (A1) at (0.5,0,-1);
            \coordinate (A2) at (-1,0,0);
            \coordinate (A3) at (.5,0,1);
            \coordinate (N)  at (0,1.5,0);
            \coordinate (S)  at (0,-1.5,0);
            \draw [fill opacity=0.8,fill=black!30] (N) -- (S) -- (A3) -- cycle; 
            \draw [fill opacity=0.8,fill=black!60] (A3) -- (A1) -- (N) -- cycle;
            \draw [fill opacity=0.8,fill=black!80] (A3) -- (A1) -- (S) -- cycle;
            \end{scope}

            \begin{scope}[shift={(-0.2,0,+1.0)}]
            \coordinate (A1) at (0.5,0,-1);
            \coordinate (A2) at (-1,0,0);
            \coordinate (A3) at (.5,0,1);
            \coordinate (N)  at (0,1.5,0);
            \coordinate (S)  at (0,-1.5,0);
            \draw [dashed] (N) -- (S);
            \draw [fill opacity=0.8,fill=black!20] (A3) -- (A2) -- (N) -- cycle;
            \draw [fill opacity=0.8,fill=black!30] (A3) -- (A2) -- (S) -- cycle;
            \end{scope}

        \end{scope}

    \end{tikzpicture}
    \caption{The two triangulations of $C([4],3)$ corresponding to the upper
    boundary $\U([4],3)$ and lower boundary $\L([4],3)$ of $|\Delta^4|$,
    respectively.}%
    \label{fig:3-2}
\end{figure}

Of crucial relevance for much of what follows will be a partial ordering on the
collection of $d$-simplices contained in the cyclic polytope $C([n],d)$ which we will
now discuss. For $n \ge d \ge 1$, consider the set $S([n],d) \subset \P([n])$
of subsets of cardinality $d+1$. For $I \in S([n],d)$ we denote 
\[
    \Delta^I := \langle \{I\} \rangle \subset \P^*([n]).
\]
The geometric realization $|\Delta^I| = C(I,d)$ in $\RR^d$ can be viewed as a
subsimplex of the cyclic polytope $C([n],d)$. As explained in Example
\ref{exa:pachner}, the boundary of $|\Delta^I|$ decomposes into lower and upper
hemisphere
\[
   \partial |\Delta^I| = |(\Delta^I)^-| \cup |(\Delta^I)^+|
\]
comprised of the even and odd faces of $\Delta^I$, respectively.
For $I,J \in S([n],d)$ with $|I \cup J| = d+2$, we declare $I \prec J$ if
\[
    I \cap J \in (\Delta^I)^+ \cap (\Delta^J)^-.
\]
Note that this means, that the affine hyperplane $H$ spanned by $I \cap J$
intersects $|\Delta^I|$ and $|\Delta^J|$ in the common face $F = |\Delta^{I \cap
J}|$ with $|\Delta^I| \setminus F$ lying below and $|\Delta^J| \setminus F$
lying above $H$. Within this context, we have the following important result which we refer to
as Rambau's lemma (cf. \cite[Corollary 5.9]{rambau}): 

\begin{lem}
    \label{lem:rambau}
    The transitive closure of the relation $\preceq$ on the set $S([n],d)$ defines a
    partial order. In particular, there are no cycles of simplices related by $\prec$.
\end{lem}
\begin{proof}
    To prove the statement, we construct an order preserving map from
    $S([n],d)$ into a totally ordered set (which implies the claim). Namely,
    define the map
    \[
        \epsilon: S([n],d) \to \{o,*,e\}^{n+1}, I \mapsto (\epsilon_0, \epsilon_1, ..., \epsilon_n)
    \]
    where 
    \[
        \epsilon_k = \begin{cases}
            o & \text{if $k$ is an odd gap of $I$,}\\
            * & \text{if $k \in I$,}\\
            e & \text{if $k$ is an even gap of $I$.}
        \end{cases}
    \]
    Interpreting the elements of $\{o,*,e\}^{n+1}$ as words of length $n+1$ in
    the alphabet $\{o,*,e\}$, we equip it with the lexicographic order induced
    by $o < * < e$. It is then immediate to verify that the map $\epsilon$ is
    order preserving. 
\end{proof}

\section{Higher Segal objects}%
\label{sec:higher_segal_objects}

We have now collected the necessary preliminaries to define the objects of
interest in this article: the higher Segal objects. We will freely use the
language of $\infty$-categories with \cite{lurie:htt} as a standard reference.
Let $\C$ be an $\infty$-category and let 
\[
    X: \Delta^{\op} \to \C
\]
be a simplicial object in $\C$. 

\begin{rem}
    \label{rem:custom}
    A (coherent) diagram valued in an $\infty$-category $\C$ indexed by an
    ordinary category $I$ is formally defined as a functor of
    $\infty$-categories (i.e. a map of simplicial sets) $\N(I) \to \C$.
    To avoid clumsy notation, we will often leave the passage to the nerve
    implicit and simply write $I \to \C$ instead.
\end{rem}

\begin{rem}
    \label{rem:convenient}
    It will be (notationally) convenient to enlarge the parametrizing category
    $\Delta$ slightly as follows: Let $\widetilde{\Delta}$ denote the category of
    finite nonempty subsets of $\NN$, equipped with the linear order induced from
    $\NN$, and weakly monotone maps. There is a canonical functor given on objects
    by
    \begin{equation}
        \label{eq:nonstandard}
        p: \widetilde{\Delta} \to \Delta, I = \{i_0 < i_1 < ... < i_n\}  \mapsto [n].
    \end{equation}
    We may thus extend the domain of the simplicial object $X$ from $\Delta$ to
    $\widetilde{\Delta}$ by pulling back along $p$. Typically, this will be left
    implicit and we simply write $X_I$ to denote $X_{p(I)}$. 
\end{rem}

\begin{con}
    \label{con:descent}
    Let $X: \Delta^\op \to \C$ be a simplicial object valued in an
    $\infty$-category $\C$, let $n \ge 0$ and let $K \subset \P^*([n])$ be an
    abstract simplicial complex on the set $[n]$. We may consider $K$ as a
    poset with respect to inclusion of subsets and form the poset $K^{\rhd}$ by
    adjoining a maximal element to $K$. We obtain functors
    \[
        K^{\rhd} \overset{q}{\to} \widetilde{\Delta} \overset{p}{\to} \Delta
    \]
    where $p$ is the functor from \eqref{eq:nonstandard} and $q$ is defined via
    $I \mapsto I$ on $K$ while the maximal element of $K^{\rhd}$ is mapped to
    $[n]$. Pulling back the simplicial object $X$ along $p \circ q$, we 
    obtain a cone
    \begin{equation}
        \label{eq:Kcone}
        (K^{\rhd})^{\op} \to \C
    \end{equation}
    over the base diagram $X|K^{\op}$ with tip $X_n$. Therefore, assuming that the limit 
    \[
        X_K := \lim X|K^{\op}
    \]
    exists in $\C$, we obtain a canonical map
    \begin{equation}
        \label{eq:limitmap}
        X_n \lra X_{K}.
    \end{equation}
\end{con}

\begin{defi}
    \label{defi:higher_segal}
    Let $\C$ be an $\infty$-category with finite limits and let $X: \Delta^{\op} \to \C$ be a
    simplicial object in $\C$. Let $d \ge 1$. Then $X$ is called 
    \begin{enumerate}
        \item\label{it:lower} lower $d$-Segal if, for every $n > d$, the map $X_n \to X_{\L([n],d)}$ from \eqref{eq:limitmap} for $K = \L([n],d)$ is an equivalence in $\C$.
        \item\label{it:upper} upper $d$-Segal if, for every $n > d$, the map $X_n \to X_{\U([n],d)}$ from \eqref{eq:limitmap} for $K = \U([n],d)$ is an equivalence in $\C$.
    \end{enumerate}
    We refer to the maps in \ref{it:lower} and \ref{it:upper} as {\em
    higher Segal maps} and to the simplicial objects satisfying them as
    {\em higher Segal objects in $\C$}. If $X$ is both lower and upper
    $d$-Segal then we say $X$ is {\em $d$-Segal}.
\end{defi}

\begin{rem}
    \label{rem:practice} 
    The formulation of Definition \ref{defi:higher_segal} in terms of
    $\infty$-categories subsumes several previous formulations (cf. \cite{DK12}):
    \begin{enumerate}
        \item When $\C$ is the nerve of an ordinary category,
            then the limit cones simply become limit cones in the
            sense of ordinary category theory. 
        \item When $\C$ is the $\infty$-categorical localization of a model
            category $C$, and $X$ arises from a strict simplicial object in
            $C$, then the higher Segal conditions can be expressed in terms of
            homotopy limits instead of $\infty$-categorical limits (cf.
            \cite{lurie:htt}).  
    \end{enumerate}
\end{rem}

\begin{exa}
\label{exa:lower_dim}
    Let $X: \Delta^{\op} \to \C$ be a simplicial object in an $\infty$-category
    $\C$. We give some examples of Segal conditions in low dimensions:
    \begin{enumerate}
        \item $X$ is lower $1$-Segal if, for every $n \ge 2$, the map
            \[
            \begin{tikzcd}
                X_n \ar[r] &  X_{\{0,1\}} \times_{X_{\{1\}}} X_{\{1,2\}} \times_{X_{\{2\}}}  \cdots\times_{X_{\{n-1\}}} X_{\{n-1,n\}} 
            \end{tikzcd}
            \]
            is an equivalence in $\C$. Therefore, lower $1$-Segal objects
            are the familiar {\em Segal objects}.

        \item $X$ is upper $1$-Segal if, for every $n \ge 2$, the map
            \[
                p_n: X_n \to X_{\{0,n\}}
            \]
            is an equivalence in $\C$. We claim that $X$ is essentially
            constant, i.e. equivalent to the pullback of a single object in
            $\C$ along the constant functor $\Delta \to *$.  

            To show this, observe first that, by two-out-of-three applied to
            the diagram 
            \[
            \begin{tikzcd}
                X_2  \ar[swap]{d}{d_1} & X_1 \ar[swap]{l}{s_0}\ar{dl}{\id} \\
                X_1,  &  
            \end{tikzcd}
            \]
            the degeneracy map $s_0$ is an equivalence (since $d_1 = p_2$ is an
            equivalence by assumption). Further, the degeneracy map $X_0 \to
            X_1$ is a retract of $s_0: X_1 \to X_2$ so that it is an
            equivalence as well. Thus, the two face maps $X_1 \to X_0$ are
            equivalences since they are sections of the degeneracy $X_0 \to
            X_1$. Finally, by composing this equivalence with the equivalence
            $p_n$, we obtain that, for every $n \ge 1$, the map
            \[
                X_n \to X_{\{0\}} = X_0
            \]
            corresponding to the inclusion of $\{0\}$ in $[n]$ is an
            equivalence. But this finally implies that the total degeneracy map
            \[
                X_0 \to X_n
            \]
            corresponding to the unique map $[n] \to [0]$ is an equivalence. A
            further application of two-out-of-three implies that $X$ maps
            all morphisms in $\Delta$ to equivalences in $\C$. The claim now
            follows, since the geometric realization of $\Delta$
            (i.e. the $\infty$-categorical localization along all edges) is
            contractible.

        \item The generating simplices of $\L([n],2)$ are $\{0,1,2\},\{0,2,3\},
            \ldots, \{0,n-1,n\}$ so that, by a basic cofinality argument, we
            observe that $X$ is lower $2$-Segal if and only if, for every $n
            \ge 3$, the map
            \[
                X_n \to X_{\{0,1,2\}} \times_{X_{\{0,2\}}} X_{\{0,2,3\}} \times_{X_{\{0,3\}}} \cdots \times_{X_{\{0,n-1\}}} X_{\{0,n-1,n\}}
            \]
            is an equivalence, i.e. if and only if the simplicial object $P^{\lhd}X$ 
            obtained from $X$ by pulling back along the join functor
            \[
                \Delta \lra \Delta, [n] \mapsto [0] \ast [n]
            \]
            is lower $1$-Segal. Similarly, $X$ is upper $2$-Segal if and only if the
            pullback $P^{\rhd}X$ of $X$ along the join functor
            \[
                \Delta \lra \Delta, [n] \mapsto [n] \ast [0]
            \]
            is lower $1$-Segal. By the path space criterion of \cite{DK12,gkt} it
            follows, that $X$ is $2$-Segal in the sense of \cite{DK12} if and
            only if $X$ is lower and upper $2$-Segal in the sense of Definition
            \ref{defi:higher_segal}.

        \item The first lower and upper $3$-Segal conditions are equivalent to the maps
            \begin{equation}
                \label{eq:first_3segal}
                X_4 \to 
                \lim \left\{
                \begin{tikzcd}
                    X_{\{0,1,2,3\}}  \ar{rr}\ar{dr}\ar{drr} & &  X_{\{1,2,3\}}\ar{d} & &  \ar{ll}\ar{dll} X_{\{1,2,3,4\}} \ar{dl}\\
                                                    & X_{\{0,1,3\}}  \ar{r}   & X_{\{1,3\}} & \ar{l} X_{\{1,3,4\}} &\\
                                                    & & \ar{ul} X_{\{0,1,3,4\}}\ar{ur} \ar{u} & &
                \end{tikzcd}
                \right\}
            \end{equation}
            and 
            \[
                X_4 \to X_{\{0,2,3,4\}}  \times_{X_{\{0,2,4\}}} X_{\{0,1,2,4\}}
            \]
            being equivalences, respectively (cf. Figure \ref{fig:3-2}). We may
            interpret the $3$-Segal maps 
            \[
            \begin{tikzcd}
                X_{\U([4],3)}  & \ar{l} X_4 \ar{r} & X_{\L([4],3)} 
            \end{tikzcd}
            \]
            as comprising a $(3,2)$ ``Pachner correspondence'' (cf. \S \ref{sub:coherent_pachner_moves})
            relating the
            limits over $\U([4],3)$ and $\L([4],3)$, respectively. 
    \end{enumerate}
\end{exa}

\begin{rem}
    \label{rem:fully}
    In Theorem \ref{thm:lowerupperfull}, we will show that if $X$ is $d$-Segal
    then, for {\em every} triangulation $T \subset \P^*([n])$ of $C([n],d)$, $n
    > d$, the map 
    \[
        X_n \lra X_T
    \]
    is an equivalence in $\C$.
\end{rem}

Gale's eveness criterion paired with a basic cofinality statements lead to the
following {\em path space criteria}, generalizing similar results for $d = 2$
observed independently in \cite{DK12} and \cite{gkt}.

\begin{defi}
    \label{defi:pathspaces}
    Let $\C$ be an $\infty$-category and $X: \Delta^{\op} \to \C$ a simplicial object in $\C$. 
    We denote by $\Plhd X$
    the simplicial object obtained by pulling back along the functor 
    \[
        \Delta \to \Delta, [n] \mapsto [0] \ast [n]
    \]
    where $\ast$ denotes the join. $\Plhd X$ is referred to as
    the {\em initial path space} of $X$ (for Kan complexes, it models the space
    of paths with fixed starting point).
    Analogously, we denote by $\Prhd X$
    the simplicial object obtained by pulling back along the functor 
    \[
        \Delta \to \Delta, [n] \mapsto [n] \ast [0]. 
    \]
    $\Prhd X$ is referred to as the {\em final path space} of
    $X$ (for Kan complexes, it models the space of paths with fixed endpoint). 
\end{defi}

\begin{prop}
    \label{prop:pathspace}
    Let $\C$ be an $\infty$-category, $X: \Delta^{\op} \to \C$ a simplicial object, and let $d \ge 0$. 
    \begin{enumerate}[label=(\alph*)]
        \item Assume $d$ is even. Then
            \begin{enumerate}[label=(\arabic*)]
                \item $X$ is lower $d$-Segal if and only if $\Plhd X$ is lower $(d-1)$-Segal.
                \item $X$ is upper $d$-Segal if and only if $\Prhd X$ is lower $(d-1)$-Segal.
            \end{enumerate}
        \item Assume $d$ is odd. Then the following conitions are equivalent:
            \begin{enumerate}[label=(\roman*)]
                \item $X$ is upper $d$-Segal.
                \item $\Plhd X$ is upper $(d-1)$-Segal.
                \item $\Prhd X$ is lower $(d-1)$-Segal.
            \end{enumerate}
    \end{enumerate}
\end{prop}

We provide a class of examples of higher Segal spaces: Let 
\[
    \DDelta \subset \CCat
\]
denote the full $2$-subcategory of the $2$-category of small categories spanned
by the standard ordinals $[n]$, $n \ge 0$. For $0 \le k \le n$, denote by
$\DDelta([k],[n])$ the mapping category in $\DDelta$, which can be identified
with the poset of monotone $k+1$-tuples in $[n]$ with the componentwise order
induced from $[n]$.

Let $\A$ be an abelian category and define $\Snk(\A)$ to be the full subcategory
\[
    \Snk(\A) \subset \Fun(\DDelta([k],[n]), \A)
\]
consisting of those diagrams $A: \DDelta([k],[n]) \to \A$ such that
\begin{enumerate}
    \item\label{sdot:1} for every noninjective map $\tau: [k] \to [n]$, we have 
        \[
            A_{\tau} \cong 0,
        \]
    \item\label{sdot:2} for every injective map $\sigma: [k+1] \to [n]$, the sequence
        \begin{equation}
            \label{eq:ssequence}
            \begin{tikzcd}
                0 \ar{r} & A_{d_k\sigma} \ar{r} & A_{d_{k-1}\sigma} \ar{r} & ... \ar{r} & A_{d_{0}\sigma} \ar{r} & 0,
            \end{tikzcd}
        \end{equation}
        induced by the sequence
        \[
            d_k \sigma \to d_{k-1} \sigma \to  ... \to d_0 \sigma.
        \]
        in $\DDelta([k],[n])$, is exact. Note that \ref{sdot:1} along with the commutativity of the square
        \begin{equation}
            \label{eq:degenerate}
            \begin{tikzcd}
                d_i \sigma \ar{r}\ar{d} & d_{i-1}\sigma \ar{d}\\
                s_{i-2} d_{i-2} d_i\sigma\ar{r}& d_{i-2}\sigma.
            \end{tikzcd}
        \end{equation}
        in $\DDelta([k],[n])$ imply that the sequence \eqref{eq:ssequence} is indeed a complex.
\end{enumerate}

Fixing $k \ge 0$, the various categories $\Snk(\A)$, $n \ge 0$, form a strict
simplicial object
\[
    \Delta^{\op} \lra \Cat
\]
in the category of categories. Postcomposing it with the nerve $\N: \Cat \to
\sSet$ and then passing to coherent nerves $\N_{\Delta}$, we obtain a coherent diagram
\[
    \Sk_{\bullet}(\A): \N(\Delta^{\op}) \lra \Cat_{\infty}
\]
which we refer to as the $k$-dimensional Waldhausen $\Se_{\bullet}$-construction
(traditionally pronounced: ``S dot construction'') of $\A$.

\begin{thm}
    \label{thm:waldhausen}
    Let $k \ge 0$. Then the simplicial object $\Sk(\A)$ is $2k$-Segal.
\end{thm}

\begin{itemize}
    \item For $k=1$, this statement was observed independently in \cite{DK12} and
        \cite{gkt}. From the perspective in \cite{DK12}, it (along with its
        many variants) served as the main motivation for introducing $2$-Segal
        spaces, namely as a means to investigate various Hall algebra
        constructions. In this case, the result is an immediate consequence of
        the path space criterion: First define the thickened nerve
        \[
            \widetilde{\N}(\A): \Delta^{\op} \lra \Cat, [n] \mapsto \Fun([n],\A).
        \]
        Then
        \begin{itemize}
            \item $\Plhd \Se^{\langle 1 \rangle}(\A)$ is equivalent to the simplicial object obtained from
                $\widetilde{\N}(\A)$ by passing, for each $n \ge 0$, to the full
                subcategory of $\Fun([n],\A)$ spanned by the composable chains of monomorphisms in $\A$. 
            \item $\Prhd \Se^{\langle 1 \rangle}(\A)$ is equivalent to the simplicial object obtained from
                $\widetilde{\N}(\A)$ by passing, for each $n \ge 0$, to the full
                subcategory of $\Fun([n],\A)$ spanned by the composable chains of epimorphisms in $\A$.
        \end{itemize}
        Both path spaces are lower $1$-Segal (i.e. classical Segal) objects so that $\Se^{\langle 1
        \rangle}(\A)$ is $2$-Segal by Proposition \ref{prop:pathspace}.

    \item For $k=2$, we observe that 
        \[
            \Plhd \Prhd \Se^{\langle 2 \rangle}(\A) = \Prhd \Plhd \Se^{\langle 2 \rangle}(\A) \simeq \widetilde{\N}(\A)
        \]
        which is lower $1$-Segal, so that $\Se^{\langle 2 \rangle}(\A)$ is
        upper $3$-Segal. By Theorem \ref{thm:fully}, it follows that
        $\Se^{\langle 2 \rangle}(\A)$ is $4$-Segal.

    \item For $k>2$, the proof still uses path space criteria, but is more involved \cite{poguntke:higher}.

\end{itemize}

\begin{rem}
    In \cite{BOORS:2segalsets, BOORS:2segalobjects}, the categorical input data
    for the $\Se_{\bullet}$-construction is generalized to a suitable class of
    ``stable'' double categories. In this context, it turns out that the
    $\Se_{\bullet}$-construction furnishes an equivalence between such double
    categories and $2$-Segal objects.  
\end{rem}

\section{Orientals}%
\label{sec:orientals}

As indicated in \S \ref{sec:cyclic_polytopes}, the analysis of the projection maps
\begin{equation}
    \label{eq:projection}
    \pi_d: C([n],d) \to C([n],d-1)
\end{equation}
offers a geometric perspective on the structure of triangulations of cyclic
polytopes. Here the term perspective can be taken quite literal -- a crucial point of 
our discussion was that two canonical triangulations of $C([n],d-1)$ become
``visible'' when looking at $C([n],d)$ from above and below with respect to the
last coordinate of $\RR^d$. Remarkably, as first pointed out in
\cite{kapranov-voevodsky}, it turns out that these two ways of looking at
simplicial complexes form above and below have categorical relevance: they
define the source and target, respectively, of a higher-dimensional composition
law. 

As a first hint towards this interpretation, note that the cyclic polytope
$C([n],n) \subset \RR^n$ can be regarded as the geometric realization of the
$n$-simplex $\Delta^n$ which we consider as an abstract simplicial complex on
the set $[n]$. As explained in Example \ref{exa:pachner}, the boundary $S :=
\partial \Delta^n$ decomposes into a lower hemisphere
\[
    S^- := (\partial \Delta^n)^- \subset \Delta^n
\]
and an upper hemisphere
\[
    S^+ := (\partial \Delta^n)^+ \subset \Delta^n.
\]
Passing to geometric realizations, by Proposition \ref{prop:features}, the
projection map $\pi$ identifies these simplicial hemispheres with
triangulations $\L([n],n-1)$ and $\U([n],n-1)$ of the cyclic
polytope $C([n],n-1)$. Under this identification, the equatorial $(n-2)$-sphere
$S^- \cap S^+$ maps to the boundary of $C([n],n-1)$.
We may now repeat this procedure in one dimension lower: The simplicial
boundary $T$ of $C([n],n-1)$ decomposes into a lower hemisphere $T^-$
and an upper hemisphere $T^+$. Interpreting all resulting abstract simplicial
complexes as subcomplexes of $\Delta^n$, we arrive at the incidence diagram
\[
\begin{tikzcd}
    & S^+ \ar{d} & \\
    T^- \ar[bend left]{ru} \ar[bend right]{rd} \ar{r} & \Delta^n & \ar{l} \ar[bend right]{lu}\ar[bend left]{ld} T^+ \\
                                               & S^- \ar{u} & 
\end{tikzcd}
\]
where
\begin{align*}
    S^- \cup S^+ & = \partial \Delta^n\\
    T^- \cup T^+ & = \partial S^- = \partial S^+ = S^- \cap S^+ 
\end{align*}
What we have thus unpacked here looks like a globular $2$-cell
\[
\begin{tikzcd}
    T^- \ar[bend left]{rr}{S^+} \ar[swap,bend right, ""{name=U}]{rr}{S^-}  & \ar[to=U, shorten <=20pt, shorten >=7pt, Rightarrow] &  T^+ 
\end{tikzcd}
\]
in a suitably defined combinatorial higher ``cobordism category'' of certain simplicial subcomplexes of
$\Delta^n$. As we will now explain, it turns out that this higher categorical
structure is a geometric incarnation of Street's oriental: the free
$\omega$-category on the $n$-simplex. 

To explain this, we first discuss some classical perspectives on strict
$(\infty,\omega)$-categories. Recall that an $\omega$-category is a set $A$
which comes equipped with a family $\{(s_n,t_n,*_n)\}_{n \in \NN}$ of maps
comprising, for every $n \in \NN$, 
\begin{itemize}
    \item an $n$-dimensional source and target map $s_n: A \to A$ and $t_n: A \to A$, respectively,
    \item a composition map
        \[
            *_n: A \times_A A \to A
        \]
        where the fiber product is defined as
        \[
            A \times_A A := \{(a,a') \in A | s_n(a) = t_n(a')\},
        \]
\end{itemize}
satisfying globularity, unitality, associativity, and commutativity constraints
(see \cite{street}). Given an $\omega$-category $A$, we may introduce, for $n \in \NN$, the set
\[
    A_{\le n} := \{a \in A | a = s_{n}(a) \} \subset A,
\]
yielding a globular set
\[
\begin{tikzcd}
    A_{\le 0} \ar[hookrightarrow]{r} & \ar[swap,bend right]{l}{s_0} \ar[bend left]{l}{t_0} A_{\le 1} \ar[hookrightarrow]{r} &\ar[swap,bend right]{l}{s_1} \ar[bend left]{l}{t_1} A_{\le 2} & \cdots
\end{tikzcd}
\]
which, when equipped with the restricted composition laws $*_n$, provides an
equivalent way to package the data of an $\omega$-category.

\begin{defi}
    \label{defi:d-admissible}
    Let $n \ge d \ge 0$. 

    \begin{enumerate}
        \item By a simplicial subcomplex of $C([n],d)$, we mean an
            abstract simplicial complex $K$ on $[n]$ consisting of subsets $\sigma \subset
            [n]$ of cardinality $\le d+1$ such that, for every $\sigma,\tau \in K$, we have 
            \[
                |\sigma| \cap |\tau| = |\sigma \cap \tau|
            \]
            where the geometric realization $|\sigma|$ denotes the convex hull of
            $\nu(\sigma) \subset \RR^d$. For a simplicial subcomplex
            $K$ of $C([n],d)$ we denote by  $|K|$ its geometric realization, defined as the
            union of the geometric realization of its simplices.

        \item For a simplicial subcomplex $K$ of $C([n],d)$ consider the
            projection map $\pi_d: \RR^d \to \RR^{d-1}$. Then $K$ is called
            $d$-admissible if, for every $x \in \pi_d(|K|)$, we have
            \[
                \{(x,t) | t \in \RR\} \cap |K| = \{x\} \times [l_x,u_x]
            \]
            for $l_x, u_x \in \RR$, $l_x \le u_x$. 
    \end{enumerate}
\end{defi}

Given a $d$-admissible subcomplex $K$ of $C([n],d)$ it is straightforward to
see that the set of points of the form $(x,l_x)$, $x \in \pi_d(|K|)$ is the
geometric realization of a subcomplex $K^- \subset K$ such that $K^-$ defines a
simplicial subcomplex of $C([n],d-1)$. Similarly, the points of the form
$(x,u_x)$ form the geometric realization of a subcomplex $K^+ \subset K$ such
that $K^+$ defines a simplicial subcomplex of $C([n],d-1)$. Generalizing the
situation of Example \ref{rem:geom_upper_lower} the simplicial complexes $K^+$
and $K^-$ can be ``seen'' when looking at $|K|$ from above and below,
respectively.\\

This allows us to give the following recursive definition:

\begin{defi}
    \label{defi:admissible}
    A simplicial subcomplex $K$ of $C([n],d)$ is called {\em admissible}, if 
    \begin{enumerate}[label = \arabic *.]
        \item $K$ is $d$-admissible and 
        \item $K^-$ and $K^+$ are admissible as subcomplexes of $C([n],d-1)$.
    \end{enumerate}
\end{defi}

\begin{exa}
    \label{rem:specialadm}
    Let $K$ be a simplical subcomplex of $C([n],d)$ whose geometric realization
    defines a triangulation of $C([n],d)$ (For example, the complexes
    $\L([n],d)$ and $\U([n],d)$ from Definition \ref{def:gale_eveness}). Then
    as explained in Remark \ref{rem:geom_upper_lower}, $K$ is $d$-admissible
    and we have $K^- = \L([n],d-1)$ and $K^+ = \U([n],d-1)$. In particular, by
    induction $K$ is admissible.
\end{exa}

\begin{defi}
    \label{defi:cyclic_oriental} Let $n \ge 0$. For $n \ge d \ge 0$, define $G_d$ to be
    the set of admissible subcomplexes of $C([n],d)$ and define $G_k := G_n$
    for $k > n$.
    \begin{enumerate}
        \item Then the diagram
            \begin{equation}\label{eq:globular2}
                    \begin{tikzcd}
                        G_0 \ar[hookrightarrow]{r} & \ar[swap,bend right]{l}{s_0} \ar[bend left]{l}{t_0} G_1 \ar[hookrightarrow]{r} & \ar[swap,bend right]{l}{s_1} \ar[bend left]{l}{t_1} G_2 & \cdots
                    \end{tikzcd}
            \end{equation}
            with
            \begin{itemize}
                \item $s_d(K) = K^-$,
                \item $t_d(K) = K^+$,
                \item $u(K) = K$,
            \end{itemize}
            for $n \ge d \ge 0$,  
            and $s_d = t_d = u = \id$, for $d > n$, defines a globular set.
        \item For $K,L \in G_{d+1}$ with $s_{d}(K) = t_{d}(L)$, define the composition law
            \begin{equation}\label{eq:comp}
                K \ast_d L = K \cup L.
            \end{equation}
            Then the globular set $G_\bullet$ equipped with the composition
            laws \eqref{eq:comp} defines an $\omega$-category.
    \end{enumerate}
\end{defi}

\begin{thm}
    \label{thm:cyclic_oriental}
    The $\omega$-category $(G_{\bullet}, *_{\bullet})$ defined in Definition
    \ref{defi:cyclic_oriental} is equivalent to Street's oriental $\O_n$, i.e.
    the free $\omega$-category on the $n$-simplex. 
\end{thm}

There are different ways to prove Theorem \ref{thm:cyclic_oriental}: one
possibility is to verify that the category $\O_n$ is freely generated by the
collection of simplices $\Delta^I$, $I \subset [n]$ (Each such a simplex gives
rise to a $(|I|-1)$-dimensional morphism from $(\Delta^I)_{-}$ to $(\Delta^I)_+$). 

The following lemma (cf. \cite[Proposition 5.11]{rambau}) is the key step in
this method of proof. It is of independent interest to us, as it will be used
in our investigation of higher Segal conditions and their relation to
correspondence categories. 

\begin{lem}
    \label{lem:excision}
    Let $n \ge d \ge 1$. Let $K$ be an admissible subcomplex of $C([n],d)$ and
    assume that $K^- \neq K^+$. Then there exists an admissible subcomplex $L$
    of $C([n],d)$ and a $d$-simplex $\Delta^I$, $I \subset [n]$, such that
    \begin{enumerate}[label = \arabic *.]
        \item $(\Delta^I)^- \subset L^+$,
        \item $L \cup \Delta^I = K$.
    \end{enumerate}
    In other words, $K$ is obtained from $L$ by stacking the simplex $\Delta^I$
    along its lower hemisphere on top of $L$. 
\end{lem}
\begin{proof}
    Since $K^- \neq K^+$, there exists at least one $d$-simplex $\Delta^J$ in
    $K$. Either $(\Delta^J)^+ \subset K^+$, or, we find another $d$-simplex $\Delta^{J'}$
    in $K$ such that
    \[
        J \prec J'.
    \]
    Since the transitive closure of $\preceq$ is a partial order by Lemma
    \ref{lem:rambau}, this chain of ``stacked'' simplices must terminate at a
    simplex $I$ in $K$ such that $(\Delta^I)^+ \subset K^+$. Removing $I$ along
    with the interior of its upper hemisphere from $K$ yields the desired
    complex $L$. It is clear that $L$ is admissible with $L^- = K^-$. 
\end{proof}

We provide a sketch of proof of Theorem \ref{thm:cyclic_oriental}: Iterated
applications of Lemma \ref{lem:excision} imply that an admissible subcomplex
$K$ of $C([n],d)$ can be written as a composite
\[
        K_1 \ast_{d-1} K_2 \ast_{d-1}\cdots \ast_{d-1} K_r  
\]
where each $K_i$ is an admissible subcomples of $C([n],d)$ containing exactly
one $d$-simplex $\sigma_i$. Each $K_i$ can in turn be written as
\[
    K_i = S_i \ast_{d-2} \sigma_i \ast_{d-2} S_i'
\]
with $S_i$ and $S'_i$ admissible subcomplexes of $C([n],d-1)$ where $S_i^+$ is
the lower equatorial hemisphere of $\sigma_i$ and $(S'_i)^-$ is the upper
equatorial hemisphere of $\sigma_i$. This implies that, given an
$\omega$-category $A$, any extension of an $\omega$-functor $G_{\le d-1} \to A$
to $G_d$ is uniquely determined by its values on $d$-simplices. Vice versa, it
is straightforward to see that an assignment on the collection of $d$-simplices,
compatible with their sources and targets, defines an extension along $G_{\le
d-1} \subset G_{\le d}$: Different choices of the factorization 
\begin{equation}\label{eq:atomic_dec}
    K_1 \ast_{d-1} K_2 \ast_{d-1}\cdots \ast_{d-1} K_r  
\end{equation}
get assigned the same values, as can be directly seen from the $2$-categorical
interchange law (Eckmann-Hilton argument) in (the $2$-category of
$(d-2)$-morphisms in) $A$. Therefore $G$ satisfies the universal property
characterizing the oriental.

\begin{rem}
    \label{rem:kapranov}
    The higher categorical relevance of cyclic polytopes was first pointed out
    in \cite{KV91} where triangulations of cyclic polytopes are identified with
    pasting schemes in the oriental. An extension to the geometric description
    of {\em all} morphisms in the oriental given by Theorem
    \ref{thm:cyclic_oriental} was established in \cite{govzman}.
\end{rem}

\section{The interplay of higher Segal conditions}%
\label{sec:the_interplay_of_higher_segal_conditions}

In this section, we outline a proof of the following result from
\cite{poguntke:higher}: 

\begin{thm}
    \label{thm:fully}
    Let $\C$ be an $\infty$-category, $X: \Delta^{\op} \to \C$ a simplicial
    object, and let $d \ge 0$. Assume that $X$ is lower $d$-Segal {\em or} upper
    $d$-Segal. Then, for every $k>d$, $X$ is $k$-Segal.
\end{thm}

\begin{lem}
    \label{lem:segal_excision}
    Let $\C$ be an $\infty$-category with finite limits and let $X:
    \Delta^{\op} \to \C$ be a lower $(d-1)$-Segal object in $\C$. Let $K,L$ as
    in Lemma \ref{lem:excision}. Then the inclusion $L \subset K$ induces an 
    equivalence
    \[
        X_K \overset{\simeq}{\lra}  X_L.
    \]
\end{lem}
\begin{proof}
    We first note that the inclusion
    \[
        i: L \cup \{I\} \subset K
    \]
    is cofinal: for every simplex $J \in K \setminus L \cup \{I\}$, the slice
    category $J/i$ consists of the single object $J \subset I$. 
    Second, we observe that the assumption that $X$ be lower $(d-1)$-Segal implies that 
    \[
        X|(L \cup \{I\})^{\op}
    \]
    is a right Kan extension of its restriction to $L^{\op}$. Indeed, the
    relevant slice category for the pointwise formula of the value of the Kan
    extension at $I$ identifies precisely with the poset $(\Delta^I)^-$ so that
    the lower $(d-1)$-Segal condition
    \[
        X_I \simeq \lim X | \L(I,d-1)^{\op}
    \]
    implies the claim.
\end{proof}

\begin{cor}
    \label{cor:d-d+1}
    Let $\C$ be an $\infty$-category with finite limits and let $X$ be a lower
    $(d-1)$-Segal object in $\C$. Then $X$ is $d$-Segal.
\end{cor}
\begin{proof}
    Let $n > d$ and $T \subset \P^*([n])$ be an abstract simplicial complex
    corresponding to a triangulation of $C([n],d)$. Then, using the notation
    from Lemma \ref{lem:excision}, we have $T$ admissible with $T^+ \neq T^-$.
    Applying Lemma \ref{lem:excision} successively, we obtain a finite sequence
    $K_1, ..., K_m$ of admissible subcomplexes such that
    \[
        K_1 \subset K_2 \subset ... \subset K_m = T
    \]
    such that 
    \begin{itemize}
        \item $K_{i+1}$ is obtained from $K_i$ by stacking a $d$-simplex on top
            of $K_i$ (more precisely $K_{i+1} = K$ and $K_{i} = L$ in Lemma
            \ref{lem:excision}),
        \item $K_1 = K_1^+ = K_1^- = T^- = \L([n],d-1)$.
    \end{itemize} 
    From Lemma \ref{lem:segal_excision} we then obtain a commutative diagram
    \[
    \begin{tikzcd}
        X_n  \ar{d}\ar[bend left=5]{drrrr}{\simeq} &&&& \\
        X_T \ar{r}{\simeq} & X_{K_{m-1}} \ar{r}{\simeq} &  X_{K_{m-2}} \ar{r}{\simeq} & ... \ar{r}{\simeq} &  X_{K_1}
    \end{tikzcd}
    \]
    in $\C$.
    Therefore, the map $X_n \to X_T$ is an equivalence so that,
    setting $T=\L([n],d)$ and $T=\U([n],d)$, we obtain, in particular, that $X$
    is lower and upper $d$-Segal.
\end{proof}

\begin{proof}[Proof of Theorem \ref{thm:fully}]
    The Theorem with hypothesis lower $d$-Segal follows by induction from
    Corollary \ref{cor:d-d+1}. For the upper $d$-Segal hypothesis, the proof
    follows precisely the same strategy, but instead interpolating between the
    triangulation $T$ and its lower boundary $T^-$ by removing simplices from
    top to bottom, we remove simplices from bottom to top to connect to $T^+$.
    The arguments adapt to this procedure mutatis mutandis. 
\end{proof}

\begin{thm}
    \label{thm:lowerupperfull}
    Let $\C$ be an $\infty$-category with finite limits, and let $X$ be a
    $d$-Segal object in $\C$. Then, for every $n > d$, and every triangulation
    $K$ of the cyclic polytope $C([n],d)$, the corresponding map
    \[
        X_n \to X_K
    \]
    from \eqref{eq:limitmap} is an equivalence in $\C$.
\end{thm}

Again, we need some basic results from the theory of cyclic polytopes. To set the stage, 
let $T$ be a triangulation of $C([n],d)$ and let $I \in S([n],d+1)$ a $d+1$-simplex
such that $(\Delta^I)^- \subset T$. Then $R = T \cup \Delta^I$ is an admissible
subcomplex of $C([n],d+1)$ such that $T' = R^+$ is another triangulation of
$C([n],d)$ called the bistellar flip of $T$ along $I$. We write
\[
    T \vartriangleleft T'.
\]

\begin{lem}
    \label{lem:fully}
    Let $\C$ be an $\infty$-category with finite limits and let $X:
    \Delta^{\op} \to \C$ be a simplicial object. Let $T \vartriangleleft T'$ be
    triangulations of $C([n],d)$ related by a bistellar flip. 
    \begin{enumerate}
        \item Assume that $X$ is lower $d$-Segal. Then the morphism 
            \[
                 X_R \to X_T 
             \]
            is an equivalence in $\C$.
        \item Assume that $X$ is upper $d$-Segal. Then the morphism 
            \[
                 X_R \to X_{T'}
            \]
            is an equivalence in $\C$.
    \end{enumerate}
\end{lem}
\begin{proof}
    This is a special case of Lemma \ref{lem:segal_excision}.
\end{proof}

\begin{proof}[Proof of Theorem \ref{thm:lowerupperfull}]
    Suppose that $T$ is a triangulation of $C([n],d)$. Then by \cite[Theorem
    1.1(i)]{rambau} there exists a sequence of bistellar flips
    \[
        T \vartriangleleft T_1 \vartriangleleft ... \vartriangleleft \U([n],d)
    \]
    connecting $T$ with the upper triangulation of $C([n],d)$. Combining Lemma \ref{lem:fully} with the equivalence
    \[
        X_n \overset{\simeq}{\lra}  X_{\U([n],d)}
    \]
    we deduce that $X_n \lra  X_T$ is an equivalence as well.
\end{proof}

\section{Higher correspondences}%
\label{sec:correspondences}

The goal of this section will be to establish an interpretation of the higher
Segal conditions in terms of higher correspondence categories. 
Our discussion is based on the classical simplicial combinatorics of
barycentric subdivision. Consider the functor
\[
    \Delta \to \sSet, [n] \mapsto \N(\P^*([n])^{\op})
\]
where $\P^*([n])$ denotes the poset of nonempty subsets of $[n]$. Its left Kan
extension along the Yoneda embedding $\Delta \to \sSet$ is the left adjoint of
an adjunction (cf. \cite{kan:css} for a classical appearance)
\[
    \sd: \sSet \lra \sSet: \co_{\infty}.
\]
Note that this is different from the {\em edgewise} subdivision, see Remark \ref{rem:edgewise} below.

\begin{con}
    \label{con:cor}
    Let $\C$ be an $\infty$-category with limits. Then the vertices of $\co_{\infty}(\C)$ coincide
    with the vertices of $\C$ while an edge in $\co_{\infty}(\C)$ corresponds to a
    diagram
    \[
    \begin{tikzcd}
        x_0 & \ar{l} x_{01} \ar{r} & x_1
    \end{tikzcd}
    \]
    in $\C$.
    A $2$-simplex in $\co_{\infty}(\C)$ corresponds to a coherent diagram
    \begin{equation}
        \label{eq:cor}
        \begin{tikzcd}
               &    & x_{1} &  & \\
               & x_{01}\ar{ur}\ar{dl} & x_{012} \ar{d}\ar{u}\ar{l}\ar{r}\ar{dll}\ar{drr}  & x_{12}\ar{ul}\ar{dr} & \\
            x_{0} & & \ar{ll} \ar{rr} x_{02} & & x_{2}
        \end{tikzcd}
    \end{equation}
    in $\C$. Note that the diagram \eqref{eq:cor} gives rise to a diagram
    \[
    \begin{tikzcd}
        & \ar{dl}\ar{dr} x_{01}\times_{x_1} x_{12} & \\
        x_0 & \ar{l}x_{012}\ar{r}\ar{d}\ar{u} &  x_2\\
            & x_{02}\ar{ur}\ar{ul} & 
    \end{tikzcd}
    \]
    which we may interpret as a globular $2$-cell in a higher category of
    correspondences: a correspondence ($2$-morphism) between the correspondences ($1$-morphisms)
    \[
    \begin{tikzcd}
        x_0 & \ar{l} x_{01}\times_{x_1} x_{12} \ar{r} & x_2
    \end{tikzcd}
    \]
    and
    \[
    \begin{tikzcd}
        x_0 & \ar{l} x_{02} \ar{r} & x_2.
    \end{tikzcd}
    \]
    In a similar fashion, every $n$-simplex 
    \[
        x: \P^*([n])^{\op} \to \C
    \]
    gives rise to an $n$-dimensional globe whose $d$-dimensional equator is given by
    \begin{equation}
        \label{eq:globular}
        \begin{tikzcd}
            & \ar{dl}\ar{dr} \lim x|\L([n],d) & \\
            \lim x|\L([n],d-1) &  & \lim x|\U([n],d-1) \\
                & \lim x|\U([n],d) \ar{ur}\ar{ul} & 
        \end{tikzcd}
    \end{equation}
    while the top $n$-dimensional cell is given by 
    \[
    \begin{tikzcd}
        & \ar{dl}\ar{dr} \lim x|\L([n],n-1) & \\
        \lim x|\L([n],n-2) & x_{\{0,1,...,n\}}\ar{r}\ar{l}\ar{u}\ar{d} & \lim x|\U([n],n-2) \\
            & \lim x|\U([n],n-1).\ar{ur}\ar{ul} & 
    \end{tikzcd}
    \]
    The legs of the correspondence
    \begin{equation}
        \label{eq:topcell}
        \begin{tikzcd}
            \lim x|\L([n],n-1) & x_{\{0,1,...,n\}}\ar{r}{u}\ar[swap]{l}{l} & \lim x|\U([n],n-1)
        \end{tikzcd}
    \end{equation}
    are precisely the upper, resp. lower, $(n-1)$-Segal maps induced by the
    higher Segal cones of Definition \ref{defi:higher_segal}.
\end{con}

\begin{defi}
    \label{defi:thin}
    Let $\C$ be an $\infty$-category with limits and let $\sigma_x$ be an $n$-simplex
    of $\co_{\infty}(\C)$ corresponding to a diagram
    \begin{equation}
        \label{eq:barysimplex}
        x: \P^*([n])^{\op} \to \C.
    \end{equation}
    Then $\sigma_x$ is called {\em lower} (resp. {\em upper}) {\em thin} if the map $l$ (resp. $u$) in 
    \eqref{eq:topcell} is an equivalence in $\C$. We call $\sigma_x$ {\em thin} if it is both lower and upper thin.
\end{defi}

The following result is due to J. Gödicke, Q. Ho, and W. Stern, and a detailed
proof will appear in a forthcoming paper by these authors. Here, we provide a sketch
of an argument emphasizing the connection to orientals. 

\begin{prop}
    \label{prop:complicial}
    Let $\C$ be an $\infty$-category with limits. Then the simplicial set
    $\co_{\infty}(\C)$, stratified with the thin simplices from Definition
    \ref{defi:thin}, is a complicial set (cf. \cite{verity:complicial, riehl:complicial}).
\end{prop}
\begin{proof}
    We set $\D := \co_{\infty}(\C)$ and sketch how to verify the filling
    conditions that needs to be satisfied (cf. Definition 2.13 in
    \cite{riehl:complicial}). Let $n \ge 3$, $0 \le i \le n$, and let 
    \[
        \tau: \Lambda^n_i \to \D
    \]
    be a horn in $\D$ such that all simplices in $\Lambda^n_i$ containing 
    the vertices $\{i-1,i,i+1\} \cap [n]$ are thin. We denote by $\P^n_i$ the poset of
    nondegenerate simplices of $\Lambda^n_i$. The horn $\tau: \Lambda^n_i \to \D$ then
    corresponds to a diagram
    \[
        x: (\P^n_i)^{\op} \to \C.
    \]
    Consider the full inclusions of posets
    \[
        (\P^n_i)^{\op} \overset{\alpha}{\hra} (\P^n_i)^{\op} \cup \{[n]\} \overset{\beta}{\hra} \P^*([n])^{\op}.
    \]
    We define
    \begin{enumerate}
        \item $x': (\P^n_i)^{\op} \cup \{[n]\} \to \C$ to be the right Kan extension of $x$ along $\alpha$, so that 
            \[
                x'([n]) \simeq \lim_{I \in \P^n_i} x_I,
            \]
        \item $x'': \P^*([n])^{\op} \to \C$ to be the left Kan extension of $x'$ along $\beta$, so that
            \[
                x''(\partial_i) \simeq x'([n]).
            \]
    \end{enumerate}
    Then $x''$ corresponds to an $n$-simplex $\sigma: \Delta^n \to \D$
    extending $\tau$. Using the imposed thinness conditions, it can now be
    shown that the simplex $\sigma$ is indeed thin. 

    The second type of filling condition a complicial set needs to satisfy is the
    following: For $\sigma: \Delta^n \to \D$ is an $n$-simplex and $0 \le i \le
    n$ such that 
    \begin{enumerate}[label= \arabic *.]
        \item $\sigma$ maps all subsimplices containing $\{i-1,i,i+1\} \cap [n]$ to thin simplices, and
        \item $\sigma$ maps the faces $\partial_{i-1} \Delta^n$ and $\partial_{i+1} \Delta^n$ to thin simplices,
    \end{enumerate} 
    then the simplex $\sigma \circ \partial_{i}$ is thin.

    To verify this condition suppose that the face $\partial_i \Delta^n$ belongs to the
    lower hemisphere $\L([n],n-1)$ of $\Delta^n$ and let 
    \[
        x: \P^*([n])^{\op} \to \C
    \]
    be the diagram corresponding to $\sigma$. We consider the diagram
    \begin{equation}
        \begin{tikzcd}
            & \ar[swap]{dl}{s_l}\ar{dr}{t_l} \lim x|\L([n],n-1) & \\
            \lim x|\L([n],n-2) & x_{\{0,1,...,n\}}\ar{r}\ar{l}\ar{u}{s}\ar[swap]{d}{t} & \lim x|\U([n],n-2) \\
                               & \lim x|\U([n],n-1).\ar[swap]{ur}{s_u}\ar{ul}{t_u} & 
        \end{tikzcd}
    \end{equation}
    where $s$ and $t$ are equivalences, due to the thinness of $\sigma$. By
    decomposing $\U([n],n-1)$ as in \eqref{eq:atomic_dec} and using the
    assumption that all faces in $\U([n],n-1)$ are thin, we deduce that $t_u$
    and $s_u$ are equivalences so that, by two/three, the maps $s_l$ and $t_l$
    are equivalences as well. Now decomposing $\L([n],n-1)$ as in
    \eqref{eq:atomic_dec} as well, and letting $K_i$ denote the subcomplex
    containing the face $\partial_i$, we deduce, again by two/three, that the
    legs in the correpondence
    \[
        \begin{tikzcd}
            x|(K_i)^-  &\ar{l}{\simeq}   x|K_i  \ar{r}{\simeq} & x|(K_i)^+ 
        \end{tikzcd}
    \]
    are equivalences.
    Finally, the imposed thinness conditions on $\sigma$ imply that the latter
    correspondence is equivalent to the correspondence
    \[
        \begin{tikzcd}
            x|(\partial_i \Delta^n)^-  &\ar{l}   x|\partial_i \Delta^n  \ar{r}{\simeq} & x|(\partial_i \Delta^n)^+ 
        \end{tikzcd}
    \]
    so that its legs must be equivalences as well, showing that the simplex
    $\sigma \circ \partial_i$ is thin.
\end{proof}

According to Proposition \ref{prop:complicial}, we may interpret
$\co_{\infty}(\C)$ as a model for the $(\infty,\omega)$-category higher
correspondences in $\C$ (cf. \cite{loubaton:complicialmodel}). Various
interesting truncations of $\co_{\infty}(\C)$ will be relevant for us:
\begin{itemize}
    \item Denote by $\co_n(\C) \subset \co_{\infty}(\C)$ the simplicial subset
        consisting of those simplices all of whose $d$-subsimplices for $d >
        n$ are thin. Then $\co_n(\C)$ models an $(\infty,n)$-category of higher
        correspondences where all correspondences above globular dimension $n$
        (cf. \eqref{eq:globular}) are invertible. 
    \item Denote by $\co_n^l(\C) \subset \co_{n}(\C)$ the simplicial subset
        consisting of those simplices all of whose $(n-1)$-subsimplices are
        lower thin. Then $\co_n^l(\C)$ models an $(\infty,n)$-category of higher
        correspondences where both legs of all correspondences above globular
        dimension $n$ (cf. \eqref{eq:globular}) are invertible and, in
        addition, the lower legs of the correspondences in globular dimension
        $n-1$ are invertible. In particular, by inverting the lower leg, we may
        interpret such a correspondence as a morphism from the lower to the
        upper hemisphere of dimension $n-2$. 
    \item Similarly, denote by $\co_n^u(\C) \subset \co_{n}(\C)$ the simplicial subset
        consisting of those simplices all of whose $(n-1)$-subsimplices are
        upper thin. In this case, by inverting the upper leg of a
        correspondence in globular dimension $n-1$, we may interpret such a
        correspondence as a morphism from the upper to the lower hemisphere of
        dimension $n-2$. 
\end{itemize}

\begin{rem}
    \label{rem:edgewise}
    In \cite{DK12}, another variant of higher correspondence categories is
    used when studying $2$-Segal spaces. There, we used the {\em edgewise}
    subdivision in the form of an adjunction
    \[
        \on{tw}: \sSet \llra \sSet: \overline{\co}
    \]
    instead of the {\em barycentric} subdivision to construct an
    $\infty$-category of correspondences (aka spans). The relation between the
    two constructions is that $\overline{\co}$ models an $(\infty,2)$-category
    (all simplices in dimensions $>2$ are already thin) which is equivalent
    to the $(\infty,2)$-category $\co_2^u(\C)$ as defined here. This already
    suggests that the construction $\overline{\co}$ is too limited when
    studying higher Segal spaces in dimensions $>2$.
\end{rem}

\begin{rem}
    \label{rem:segalthin}
    Note that the Segal conditions are closely related to the thinness
    conditions in Definition \ref{defi:thin}: both are given in terms of limits
    of diagrams parameterized by the poset of facets of the upper (resp. lower)
    hemispheres of cyclic polytopes. A precise relation, along with its
    structural relevance, will be discussed in \S \ref{sec:monads}.
\end{rem}

\section{Monads}%
\label{sec:monads}

In this section, we explain how to identify higher Segal objects with certain
lax monads in higher correspondences categories. Variations of this perspective
have been established in \cite{stern:algebra, goedicke:bimodules, gale:lax}. 

In our approach, we will show how to construct from {\em any} simplicial object
in $\C$ a lax monad in the $(\infty,\omega)$-category $\co_{\infty}(\C)$ of
correspondences in $\C$. The lower and upper $n$-Segal conditions are then
responsible for yielding monads in the $(\infty,n)$-categorical truncations
$\co_n^l(\C)$ and $\co_n^u(\C)$, respectively. Finally, we then show that this
construction establishes an equivalence between suitably defined
$\infty$-categories. There doesn't seem to be a standard concept of a lax monad
available within the framework of $(\infty,\omega)$-categories so we provide a
(somewhat ad-hoc) definition suitable for our purposes. With more work, it will
be possible to put our constructions in a model independent context, but we do not
worry about this here. 

To put our to be used model of a monad in context, we recall (one possible
definition of) the notion of a monad valued in an $(\infty,2)$-category $\D$.
Informally, a monad consists of choices of
\begin{itemize}
    \item an object $x$ in $\D$,
    \item an endomorphism $M: x \to x$ in $\D$,
    \item natural transformations $\mu: M \circ M \Rightarrow M$ and $\epsilon: \id_x \Rightarrow M$,
    \item a coherent system of higher associativity and unitality constraints.
\end{itemize}
Passing to the monoidal $\infty$-category of endomorphisms $\E^{\otimes} = \D(x,x)$, a
monad can be identified with an associative algebra in $\E^{\otimes}$. Thinking
of $\E^{\otimes}$ as modelled by a coCartesian fibration $q$ over $\Delta^{\op}$,
one may define such algebra objects as sections that map convex maps in
$\Delta$ to $q$-coCartesian edges (cf. \cite[4.1]{lurie:ha}). We provide a
reformulation which can be adapted to our $(\infty,\omega)$-categorical
context: To this end, we replace $\E^{\otimes} \to \Delta^{\op}$ by a relative
subdivision as follows. 

Let $\kappa: \tot(\P^*) \to \Delta$ denote the covariant Grothendieck construction
of the functor
\begin{equation}
    \label{eq:grothendieck}
    \P^*: \Delta \to \Cat, \; [n] \mapsto \P^*([n])
\end{equation}
where $\P^*([n])$ denotes the poset of nonempty subsets of $[n]$. Further
define the functor
\begin{equation}
    \label{eq:nonstandardproj}
\lambda: \tot(\P^*) \to \Delta, ([n],I = \{0 \le i_0 < i_1 < ... < i_k \le n\}) \mapsto [k].
\end{equation}
Then define $\rsd(q): \rsd(\E^{\otimes}) \to \Delta^{\op}$ via the adjunction 
\[
    \Hom_{\Delta^{\op}}(K, \rsd(\E^{\otimes})) \cong  \Hom_{\Delta^{\op}}(K \times_{\Delta^{\op}} \tot(\P^*), \E^{\otimes})'
\]
where 
\begin{enumerate}
    \item the fiber product over $\Delta^{\op}$ is taken with respect to $\kappa$
        and equipped with the map to $\Delta^{\op}$ induced by $\lambda$,
    \item the symbol $'$ indicates that we take the subset consisting of maps 
        \[
            f: K \times_{\Delta^{\op}} \tot(\P^*) \lra \E^{\otimes}
        \]
        satisfying the following condition: for every vertex $x \in K_0$, the
        restriction of $f$ to $\{x\} \times_{\Delta^{\op}} \tot(\P^*)$ maps
        convex inclusions $I \subset J$ to $q$-coCartesian edges in $\E^{\otimes}$. 
\end{enumerate} 

\begin{exa}
    \label{exa:bary}
    A vertex of $\rsd(\E^{\otimes})$ over $[2] \in \Delta$ corresponds to a diagram in $\E^{\otimes}$
    \begin{equation}
        \scriptsize
        \label{eq:barytwo}
    \begin{tikzcd}[column sep={6em, between origins},row sep={3em, between origins}, nodes in empty cells]
        &   &   * &    & \\
        &&&&\\
        & A'\ar{uur}{!} \ar[swap]{ddl}{!} & & B'\ar[swap]{uul}{!} \ar{ddr}{!}  & \\
        &  & \ar{ul}{!} (A,B) \ar{d} \ar[swap]{ur}{!} &  & \\
        * & & \ar[swap]{ll}{!} C'\ar{rr}{!} & & *
    \end{tikzcd}
    \end{equation}
    where the edges marked $!$ are coCartesian: For example the edge $(A,B) \to
    A'$ covers the convex edge $[2] \leftarrow [1]$ given by the face map
    $\partial_2$ and corresponds to an equivalence $A \to A'$. On the other
    hand, the edge $(A,B) \to C$, which covers the nonconvex edge $\partial_1$,
    is {\em not} coCartesian and corresponds to a morphism $A \otimes B \to C$
    in $\E$. 
\end{exa}

\begin{prop}
    \label{prop:alg}
    Let $q: \E^{\otimes} \to \Delta^{\op}$ be a monoidal $\infty$-category.
    Then there is an equivalence between the $\infty$-categories of
    \begin{enumerate}
        \item sections of $q$ mapping convex edges in $\Delta$ to
            $q$-coCartesian edges in $\E^{\otimes}$ (i.e. associative algebras
            in $\E^{\otimes}$), 
        \item sections of $\rsd(q)$ mapping injective maps in $\Delta$ to
            $\rsd(q)$-coCartesian edges in $\rsd(\E^{\otimes})$. 
    \end{enumerate} 
\end{prop}
\begin{proof}
    By definition, sections of $\rsd(q)$ can be identified with diagrams 
    \[
        \begin{tikzcd}
            \tot(\P^*)^{\op} \ar{rr}\ar[swap]{dr}{\lambda^{\op}} & & \E^{\otimes} \ar{dl}\\
                                     & \Delta^{\op}. & 
        \end{tikzcd}
    \]
    The statement then follows from Lemma \ref{lem:localization}.
\end{proof}

To obtain an analogous concept of a {\em lax monad} in the
$(\infty,\omega)$-category $\co_{\infty}(\C)$, we first replace the defining
simplicial set by a simplicial $\infty$-category: While the set of
$n$-simplices $\co_{\infty}(\C)$ is defined by the formula
\[
    \co_{\infty}(\C)_n = \Hom_{\sSet}(\sd(\Delta^n), \C) 
\]
we now define
\begin{equation}
    \label{eq:co_thick}
    \underline{\co_{\infty}}(\C)_n = \underline{\Hom}_{\sSet}(\sd(\Delta^n), \C) 
\end{equation}
to be the simplicial set of maps which, since $\C$ is an $\infty$-category,
will itself be an $\infty$-category. We thus obtain a strict functor
\[
    \underline{\co_{\infty}}(\C): \Delta^{\op} \lra \sSet
\]
taking values in $\infty$-categories. We finally pass to the Grothendieck
construction (relative nerve) to obtain a coCartesian fibration
\begin{equation}
    \label{eq:thick_tot}
    \pi: \tot(\underline{\co_{\infty}}(\C)) \to \Delta^{\op}
\end{equation}
which will play an analogous role to the fibration $\rsd(q)$ above. 

\begin{defi}
    \label{defi:label}
    A {\em lax monad} in $\underline{\co_{\infty}}(\C)$ is a section of $\pi$ which maps all
    injective maps in $\Delta$ to $\pi$-coCartesian edges. 
\end{defi}

\begin{rem}
    \label{rem:model}
    Further note, that the fibration $\pi$ encodes more than simply the
    $(\infty,\omega)$-category $\co_{\infty}(\C)$ but rather a double
    categorical refinement that keeps track of morphisms in the original
    category $\C$ and leads to the right amount of ``laxness'' to make direct
    connection to higher Segal objects in $\C$. Hence our definition of monad
    should be regarded as dependent on this additional data.
\end{rem}

\begin{thm} 
    \label{thm:monad}
    Let $\C$ be an $\infty$-category with limits. 
    \begin{enumerate}
        \item\label{monad:it:1} There is an equivalence of $\infty$-categories 
           \begin{equation}
               \label{eq:monad}
               \Fun(\Delta^{\op},\C) \overset{\simeq}{\lra} \left\{\text{monads in $\underline{\co_{\infty}}(\C)$}\right\}\!,\; X \mapsto M_X
           \end{equation}
       \item\label{monad:it:2} Under the equivalence \eqref{eq:monad}, lower and upper $d$-Segal
           objects in $\C$ get identified with monads in
           $\underline{\co_{d}^l}(\C)$ and $\underline{\co_{d}^r}(\C)$,
           respectively.
    \end{enumerate}
\end{thm}
\begin{proof}
    As above, let $\tot(\P^*) \to \Delta$ denote the covariant Grothendieck construction
    of the functor \eqref{eq:grothendieck}. Unravelling the defining formula
    \eqref{eq:co_thick} of $\underline{\co}_{\infty}(\C)$, we observe that a
    section of \eqref{eq:thick_tot} corresponds to a functor
    \[
        \tot(\P^*)^{\op} \to \C.
    \]
    Pullback along the functor $\lambda$ from \eqref{eq:nonstandardproj}
    provides a functor of $\infty$-categories
    \[
        \lambda^*: \Fun(\Delta^{\op}, \C) \to \Fun(\tot(\P^*)^{\op}, \C).
    \]
    For a simplicial object $X \in \Fun(\Delta^{\op}, \C)$, the section of
    $\pi$ corresponding to $\lambda^*(X)$ maps injective maps in $\Delta$ to
    coCartesian edges: this reduces to the fact that, for $\phi: [m] \to [n]$
    injective, and $I \in \P^*([m])$, the morphism $([m],I) \to ([n],\phi(I))$
    in $\tot(\P^*)$ is mapped under $\lambda$ to an identity morphism in
    $\Delta$. The statement of \ref{monad:it:1} now follows from Lemma
    \ref{lem:localization} below. The statement of \ref{monad:it:2} is
    immediate from the discussion in \S \ref{sec:correspondences}, since higher
    Segal conditions for a simplicial object $X$ translate directly into
    thinness conditions for the monad $M_X$.
\end{proof}

\begin{lem}
    \label{lem:localization}
    The functor
    \[
        \lambda: \tot(\P^*) \to \Delta
    \]
    from the proof of Theorem \ref{thm:monad} exhibits $\Delta$ as an
    $\infty$-categorical localization of $\tot(\P^*)$ along the set of
    morphisms of the form 
    \[
        \phi: ([m],I) \to ([n],J)
    \]
    where $\phi: [m] \to [n]$ a morphism in $\Delta$ such that $\phi(I) = J$.
\end{lem}
\begin{proof}
    Let $\C$ be an $\infty$-category. We first claim that every object of
    $Y \in \Fun(\tot(\P^*),\C)$ admits a right Kan extension along $\lambda$. This
    follows from the pointwise formula: to determine the value of the right Kan
    extension at $[k] \in \Delta$, we need to compute a limit over the slice
    category $[k]/\lambda$. But this category has an initial object given by
    $([k],[k])$, so that $\lambda_*(Y)_k \simeq Y([k],[k])$. We thus obtain an
    adjunction of $\infty$-categories
    \[
        \lambda^*: \Fun(\Delta, \C) \lra \Fun(\tot(\P^*),\C): \lambda_*
    \]
    where $\lambda_*$ denotes the right Kan extension functor along $\lambda$.
    The fact that $\lambda^*$ is fully faithful now immediately follows from
    the fact that the counit $\id \to \lambda_*\lambda^*$ is an equivalence.
    Vice versa, the essential image of $\lambda^*$ is spanned by the collection
    of objects on which the unit $\lambda^* \lambda_* \to \id$ acts as an
    equivalence. Again, by the above computation of the slice categories, it
    follows that the action of the unit on $Y \in \Fun(\tot(\P^*),\C)$
    evaluated at $([n],I=\{i_0,...,i_k\})$ is the morphism 
    \[
        Y([k],[k]) \lra Y({([n],I)}).
    \]
    The fact that any functor $Y$ in the image of $\lambda^*$ needs to invert,
    more generally, morphisms of the type specified in the statement of the
    Lemma follows by two out of three.
\end{proof}

\begin{rem}
    \label{rem:unital}
    In \cite{fgkpw}, it is shown that every $2$-Segal space automatically
    satisfies certain unitality constraints which were previously required
    explicitly in the definition of so-called {\em unital} $2$-Segal spaces. It
    would be very interesting to study analogous phenomena for higher Segal
    conditions and the relevance for the lax monads of this section.
\end{rem}


\section{Higher excision}
\label{sec:higherex}

In this section, we explain a characterization of the higher Segal conditions
from \S \ref{sec:higher_segal_objects} in terms of higher-dimensional excision
conditions as they appear, for instance, in Goodwillie calculus. These results
are due to T. Walde \cite{walde:excision}.

Let $\C$ be an $\infty$-category and let
\[
    X: \Delta^{\op} \to \C
\]
be a simplicial object in $\C$.
Then it is straightforward to show that the following are equivalent:
\begin{enumerate}
    \item $X$ is a Segal object (i.e. a lower $1$-Segal object, in our sense).
    \item For every $0 \le k \le n$, the square
            \[
            \begin{tikzcd}
                X_{\{0,1,...,n\}}\ar{d} \ar{r} & X_{\{0,1,...,k\}} \ar{d}\\
                X_{\{k,k+1...,n\}} \ar{r} & X_{\{k\}} 
            \end{tikzcd}
            \]
            is a pullback square in $\C$ (following the conventions of Remark \ref{rem:convenient}).
    \item $X$ maps biCartesian squares in $\Delta$ to Cartesian squares in $\C$. 
\end{enumerate}
The latter characterization has some particular appeal, since it reformulates
the Segal conditions in term of intrinsically categorical properites of $\Delta$.
This raises the question whether there is an analogous characterization of the higher Segal conditions. 

\begin{exa}
    \label{exa:3segal}
    Unravelling the first lower $3$-Segal condition, stating that the cone
    \[
        X | (\L([4],3)^{\rhd})^{\op}
    \]
    is a limit cone in $\C$, we observe that it translates verbatim to the statement that the cube
        \begin{equation}
            \label{eq:3segcube}
            \begin{tikzcd}
                X_{\{0,1,2,3,4\}} \ar{rr}\ar{dr}\ar{dd} &  & X_{\{1,2,3,4\}}\ar{dr}\ar{dd} & \\
                                                      & X_{\{0,1,3,4\}} \ar{dd}\ar{rr} &  & \ar{dd} X_{\{1,3,4\}}\\
                X_{\{0,1,2,3\}} \ar{rr}\ar{dr} &  & X_{\{1,2,3\}}\ar{dr} & \\
                                             & X_{\{0,1,3\}} \ar{rr} & & X_{\{1,3\}}
            \end{tikzcd}
        \end{equation}
        is Cartesian (cf. \eqref{eq:first_3segal}).
    This is in analogy to the first lower $1$-Segal condition saying that the square
    \[
    \begin{tikzcd}
        X_{\{0,1,2\}}\ar{d} \ar{r} & X_{\{1,2\}} \ar{d}\\
        X_{\{0,1\}} \ar{r} & X_{\{1\}}
    \end{tikzcd}
    \]
    is a pullback square. Further, we note that the cube in $\Delta$,
    parametrizing the cube \eqref{eq:3segcube}, is a strongly biCartesian cube (i.e. all its faces are biCartesian). 
\end{exa}

\begin{thm}
    \label{thm:walde}
    Let $\C$ be an $\infty$-category with limits, let $X: \Delta^{\op} \to \C$ a simplicial
    object in $\C$, and let $d \ge 1$. Then the following are equivalent:
    \begin{enumerate}
        \item $X$ is a lower $(2d-1)$-Segal object.
        \item $X$ maps each strongly biCartesian $(d+1)$-cube in $\Delta$ to a Cartesian cube in $\C$. 
    \end{enumerate}
\end{thm}

For the proof we refer to \cite{walde:excision}.

\section{Further perspectives}
\label{sec:further}

\paragraph{Coherent Pachner moves.}%
\label{sub:coherent_pachner_moves}

Let $1 \le k \le n$, let $S \subset \partial\Delta^{n+1}$ be an abstract
simplicial complex which is the union of $k$ faces, and let $S'$ be the union
of the complementary faces. A {\em $(k,n+1-k)$-Pachner move} on a triangulated
piecewise linear $n$-manifold $M$ is the procedure of replacing a subcomplex of $M$,
identified with $S$, by $S'$.

Now let $\C$ be an $\infty$-category with limits, and let $X: \Delta^{\op} \to
\C$ be a $d$-Segal object.

We may view the datum 
\begin{equation}
    \label{eq:segal-pachner}
    \begin{tikzcd}
        S = \L([d+1],d) \ar[hook]{r} & \Delta^{d+1} & \ar[hook']{l} \U([d+1],d) = S'
    \end{tikzcd}
\end{equation}
of the lower and hemisphere as specifying a type of
$(\frac{d+1}{2},\frac{d+1}{2})$ (resp. $(\frac{d+2}{2},\frac{d}{2})$, depending
on the parity of $d$) Pachner move which we refer to as a Segal-Pachner move.
The lowest $d$-Segal maps comprise a correspondence
\[
\begin{tikzcd}
X_{S} & \ar[swap]{l}{\simeq} X_{d+1} \ar{r}{\simeq} & X_{S'}
\end{tikzcd}
\]
whose legs are equivalences, which we may interpret as the statement that the
object $X_S$ is ``invariant'' under the Pachner move which replaces $S$ by
$S'$. The higher simplices of $X$ along with the higher $d$-Segal maps provide
higher coherence data corresponding to higher associativity constraints. 

More globally, given a linearly ordered set $V$ and an abstract $n$-dimensional
simplicial complex $T$ on $V$ whose geometric realization is a manifold $M$, we
may obtain correspondences 
\[
\begin{tikzcd}
    X_{T} & \ar[swap]{l}{\simeq} X_{B} \ar{r}{\simeq}  & X_{T'}
\end{tikzcd}
\]
associated to Segal-Pachner moves thus establishing the invariance of $X_T$
(here $B$ is obtained by attaching an $n+1$-simplex to $T$). 

To turn this construction into a useful invariant of the manifold $M$, one
needs to address
\begin{enumerate}
    \item the dependence on the linear order of $V$, 
    \item the fact that we only have access to the Segal-Pachner moves (and not
        more general Pachner moves). 
\end{enumerate}

In dimension $2$, it turns out that there is a satisfying resolution of these
issues \cite{dk:triangulated,d:a1homotopy, dk:crossed}, leading to a combinatorial
framework for the construction of correspondence-valued TFTs, conceptually
perhaps culminating in a version of the cobordism hypothesis for these types of
TFTs established in \cite{stern:algebra}.

In dimensions bigger than $2$, the construction of useful analogous
higher-dimensional TFT-type manifold invariants from higher Segal objects with
additional symmetries is a very interesting open problem. 

\paragraph{Additive coefficients and $2$-categorical methods.}%
\label{sub:lax_simplicial_objects}

In \cite{DJW19}, the classical Dold--Kan correspondence is used to analyze
higher Segal conditions for simplicial abelian groups:

\begin{thm}
    \label{thm:simplicial-abelian} Let $X:\Delta^{\op} \to \A$ be a simplicial
    object valued in an abelian category $\A$ and $m \ge 1$. Then the following
    are equivalent:
    \begin{enumerate}
        \item $X$ is $2m$-Segal.
        \item $X$ has unique outer horn fillers above dimension $m$, i.e. for
            every $n > m$, the horn restriction maps
            \[
                X_n \to X_{\Lambda_0^n} \text{ and } X_n \to X_{\Lambda_n^n}
            \]
            are isomorphisms in $\A$. 
        \item The normalized chain complex $C(X)$ associated to $X$ under the
            Dold--Kan correspondence is $m$-truncated, i.e., $C(X)_{n} \cong
            0$, for $n > m$. 
    \end{enumerate}
\end{thm}

In \cite{dyck:dk}, a categorified variant of the Dold-Kan correspondence is
established, furnishing an equivalence
\[
\begin{tikzcd}
    \C: \St_{\DDelta} \lra \Ch_{\ge 0} (\St)
\end{tikzcd}
\]
between $2$-simplicial objects in the $(\infty,2)$-category $\St$ of stable
$\infty$-categories and connective chain complexes in $\St$. An interesting
feature of this correspondence is the appearance of the simplex $2$-category
$\DDelta$ defined as the full sub $2$-category of the $2$-category of
categories spanned by the standard ordinals. A $2$-simplicial object is then a
$2$-functor
\[
    \X: \DDelta^{\op} \to \St
\]
comprising an underlying ordinary simplicial object, along with presentations
of the $2$-morphisms in $\DDelta$ as natural transformations in $\St$. 

\begin{exa}
    \label{ex:2simplex}
    The $2$-morphisms among the nondegenerate subsimplices of the $2$-simplex can be depicted as follows:
    
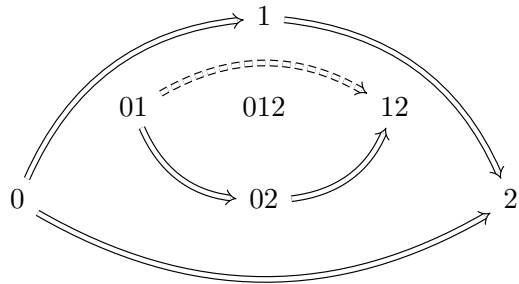
\[
    \begin{tikzcd}
               &    & {1} \ar[bend left, Rightarrow]{ddrr}&  & \\
               & {01} \ar[bend left, dashed, Rightarrow]{rr}\ar[bend right, Rightarrow]{dr} & {012}   & {12} & \\
        {0} \ar[bend left, Rightarrow]{uurr} \ar[bend right, Rightarrow]{rrrr}  & &   {02} \ar[bend right, Rightarrow]{ur} & & {2}
    \end{tikzcd}
    \]
    where the dashed arrow signifies that the composite factors through the degenerate edge $11$.
\end{exa}

While, for an ordinary simplicial object, the value of $X$ on $\Lambda_0^2$
would be defined as the limit of the diagram
\[
\begin{tikzcd}
    X_{01}\ar{dr}  & & X_{02}\ar{dl}\\
            & X_{0} & 
\end{tikzcd}
\]
the additional $2$-categorical data exhibited in Example \ref{ex:2simplex}
suggests that one may define the value of a $2$-simplicial object $\X$ on this horn
as a partially lax limit of the marked diagram 
\[
\begin{tikzcd}
    \X_{01}\ar[swap]{dr}{+} \ar{rr}  & & \X_{02}\ar{dl}{+}\\
            & \X_{0} & 
\end{tikzcd}
\]
where the markings indicate the edges over which the corresponding triangle in
the limit cone is supposed to commute up to natural $2$-isomorphism, while over
unmarked edges, we only require commutativity up to a possibly noninvertible
$2$-morphism. 

In work in progress, we establish a suitable formalism to analyze the
corresponding horn filling conditions, and establish the following result:

\begin{thm}
    \label{thm:2simplicial}
    Let $\X \in \St_{\DDelta}$ be a $2$-simplicial stable $\infty$-category.
    \begin{enumerate}
        \item Then $\X$ is an {\em outer Kan complex} in the sense that for every $n \ge 1$, 
            \begin{enumerate}
                \item the map $\X_n \lra \X_{\LLambda_0^n}$ admits a fully faithful right adjoint (left localization),
                \item the map $\X_n \lra \X_{\LLambda_n^n}$ admits a fully faithful left adjoint (right localization).
            \end{enumerate}
        \item The following are equivalent:
            \begin{enumerate}
                \item for every $n > m$, the map $\X_n \overset{\simeq}{\lra} \X_{\LLambda_0^n}$ is an equivalence,
                \item for every $n > m$, the map $\X_n \overset{\simeq}{\lra} \X_{\LLambda_n^n}$ is an equivalence,
                \item the categorified normalized chain complex $\C(\X)$ is $m$-truncated,
                \item $\X|\Delta$ is $2m$-Segal.
            \end{enumerate}
    \end{enumerate}
\end{thm}
    
\begin{rem}
    \label{rem:general}
    We expect that the categorified Dold-Kan correspondence along with Theorem
    \ref{thm:2simplicial} generalize from $\St$ to lax additive
    $(\infty,2)$-categories which are idempotent complete in a suitable sense. 
    It would also be very interesting to establish general relations between
    $2$-horn filling conditions and higher Segal conditions beyond this lax
    additive context. 
\end{rem}

\addcontentsline{toc}{section}{References}
\bibliographystyle{alpha}
\bibliography{refs}

\end{document}

%% file: preamble.tex
\usepackage[utf8]{inputenc}

\usepackage{float}
\usepackage{tikz-cd}
\usepackage{amsxtra}
\usepackage{amsmath}
\usepackage{amssymb}
\usepackage{amsfonts}
\usepackage[all,2cell]{xy}
\usepackage{amsthm}
\usepackage{enumitem}
\usepackage{hyperref}
\usepackage{subfigure}
\usepackage{euscript}
\usepackage{todonotes}
\usepackage[margin=3cm]{geometry}
\usepackage[bbgreekl]{mathbbol} 
\usepackage{adjustbox}

\numberwithin{equation}{section}

\newtheorem{thm}[equation]{Theorem}
\newtheorem*{thm*}{Theorem}
\newtheorem{cor}[equation]{Corollary}
\newtheorem{lem}[equation]{Lemma}
\newtheorem{prop}[equation]{Proposition}

\theoremstyle{definition}
\newtheorem{defi}[equation]{Definition}
\newtheorem{rem}[equation]{Remark}
\newtheorem{exa}[equation]{Example}
\newtheorem*{exa*}{Example}
\newtheorem{con}[equation]{Construction}

\usepackage{xcolor}%
\hypersetup{%
    colorlinks,%
    linkcolor={red!50!black},%
    citecolor={blue!50!black},%
    urlcolor={blue!80!black}%
}

\def\on{\operatorname}
\def\op{{\on{op}}}

\def\id{\on{id}}

\def\Fun{\on{Fun}}
\def\lim{\on*{lim}}

\def\tot{\on*{tot}}

\def\Hom{\on{Hom}}

\def\N{\on{N}}

\def\P{\mathcal{P}}
\def\L{\mathcal{L}}
\def\U{\mathcal{U}}

\def\Prhd{\on{P}^{\rhd}\!}
\def\Plhd{\on{P}^{\lhd}\!}

\def\hra{\hookrightarrow}
\def\lra{\longrightarrow}

\def\llra{\longleftrightarrow}

\def\O{\mathbb{O}}

\DeclareMathSymbol\DDelta\mathord{bbold}{"01}
\DeclareMathSymbol\GGamma\mathord{bbold}{"00}
\DeclareMathSymbol\LLambda\mathord{bbold}{"03}
\DeclareMathSymbol\SSigma\mathord{bbold}{'117}

\def\C{\EuScript{C}}

\def\Cat{\on{Cat}}
\def\CCat{\on{\mathbb{C}at}}

\def\St{{\EuScript S}t}
\def\Set{{\on{Set}}}
\def\sSet{ {\Set_{\Delta}}}

\def\Ch{\on{Ch}}
\def\sd{\on{sd}}
\def\rsd{\on{rsd}}
\def\co{\on{co}}

\def\Snk{{\EuScript S}_n^{\langle k \rangle}}
\def\Sk{{\EuScript S}^{\langle k \rangle}}
\def\Se{{\EuScript S}}

\def\NN{{\mathbb{N}}}

\def\A{{\EuScript A}}
\def\X{{\EuScript X}}

\def\D{{\EuScript D}}
\def\E{{\EuScript E}}

\def\RR{{\mathbb R}}

\def\U{ {\EuScript U}}

\setlist[enumerate,1]{label=(\arabic{*})}
\setlist[enumerate,2]{label=(\alph{*})}
\setlist[enumerate,3]{label=(\roman{*})}